%% file: sepstabilizer2.tex
\documentclass[10pt]{amsart}

\usepackage{subfigure}
\usepackage{amssymb}
\usepackage{graphicx}
\usepackage{palatino}

\usepackage{url}
\usepackage{cite}

\input xy
\xyoption{all}
\xyoption{2cell}
\UseComputerModernTips
\UseTwocells

\numberwithin{equation}{section}
\numberwithin{figure}{section}

\makeatletter
\newtheorem*{rep@theorem}{\rep@title}
\newcommand{\newreptheorem}[2]{
\newenvironment{rep#1}[1]{
\def\rep@title{#2 \ref{##1}}
\begin{rep@theorem}}
{\end{rep@theorem}}}
\makeatother

\newtheorem{theorem}{Theorem}[section]
\newreptheorem{theorem}{Theorem}
\newreptheorem{proposition}{Proposition}
\newreptheorem{corollary}{Corollary}
\newtheorem{lemma}[theorem]{Lemma}
\newtheorem{proposition}[theorem]{Proposition}
\newtheorem{corollary}[theorem]{Corollary}

\newtheorem{remark}[theorem]{Remark}

\theoremstyle{definition}
\newtheorem{definition}[theorem]{Definition}
\newtheorem{example}[theorem]{Example}

\newcommand{\C}{{\mathbb{C}}}

\newcommand{\Z}{{\mathbb{Z}}}

\newcommand{\R}{{\mathbb{R}}}

\renewcommand{\P}{{\mathbb{P}}}

\newcommand{\calX}{\mathcal{X}}

\newcommand{\id}{\mathrm{id}}

\newcommand{\into}{\hookrightarrow}

\newcommand{\dual}{\star}
\newcommand{\DG}{\mathrm{DG}}

\newcommand{\CP}{\C\mathrm{P}}

\DeclareMathOperator{\coker}{coker}
\DeclareMathOperator{\modulo}{mod}

\DeclareMathOperator{\Stab}{Stab}

\DeclareMathOperator{\Hom}{Hom}

\DeclareMathOperator{\Top}{\mathsf{Top}}

\DeclareMathOperator{\Tor}{Tor}

\DeclareMathOperator{\im}{im}

\DeclareMathOperator{\Diff}{\mathsf{Diff}}

\usepackage{multicol}
\usepackage{color}

\definecolor{gold}{rgb}{0.85,.66,0}
\definecolor{cherry}{rgb}{0.9,.1,.2}
\definecolor{burgundy}{rgb}{0.8,.2,.2}
\definecolor{orangered}{rgb}{0.85,.3,0}
\definecolor{orange}{rgb}{0.85,.4,0}
\definecolor{olive}{rgb}{.45,.4,0}
\definecolor{lime}{rgb}{.6,.9,0}
\definecolor{green}{rgb}{.2,.7,0}
\definecolor{grey}{rgb}{.4,.4,.2}
\definecolor{brown}{rgb}{.4,.3,.1}

\newcommand{\und}{\underline}

\begin{document}

\title{Global quotients among toric Deligne-Mumford stacks}

\author{Megumi Harada}
\address{Department of Mathematics and
Statistics\\ McMaster University\\ 1280 Main Street West\\ Hamilton, Ontario L8S4K1\\ Canada}
\email{Megumi.Harada@math.mcmaster.ca}
\urladdr{\url{http://www.math.mcmaster.ca/Megumi.Harada/}}
\thanks{MH is partially supported by an NSERC Discovery Grant,
an NSERC University Faculty Award, and an Ontario Ministry of Research
and Innovation Early Researcher Award.}

\author{Derek Krepski}
\address{Department of Mathematics \\ 
University of Western Ontario\\ 
Middlesex College \\
London, Ontario N6A 5B7\\ Canada}
\email{dkrepski@uwo.ca}


\date{\today}



\begin{abstract}

This work characterizes global quotient stacks---smooth stacks associated to a finite group acting a manifold---among smooth quotient stacks $[M/G]$, where $M$ is a smooth manifold equipped with a smooth proper action by a Lie group $G$.  The characterization is described in terms of the action of the connected component $G_0$ on $M$ and is related to (stacky) fundamental group and covering theory.  This characterization is then applied to smooth toric Deligne-Mumford stacks, and global quotients among toric DM stacks are then characterized in terms of their associated combinatorial data of stacky fans.  

%
%

\end{abstract}

\maketitle


\section{Introduction}

This note has two parts: first, we consider topological properties of smooth orbifolds that
describe obstructions to being a \emph{global (resp. discrete)
  quotient}\footnote{In the literature, global quotients are also called \emph{good} or
\emph{developable} orbifolds.} ---i.e. equivalent (in a sense made more precise below) to an orbifold associated 
to a finite (resp. discrete) group action on a smooth manifold, and second, we apply our general results on global quotients
to toric Deligne-Mumford stacks. 
Recall that an orbifold structure,
considered from the classical viewpoint (originating in the work of
Satake \cite{Satake:1956}), can be described via local charts, each of
which are quotients $U/\Gamma$ of a linear action of a finite group
$\Gamma$ on an open subset $U$ of Euclidean space. 
The point is that for a general orbifold, these local charts do not necessarily `patch together'
into a global finite group action on a smooth manifold, but for a
global quotient, they do. 

It is worth emphasizing the following few points at the outset. 
Firstly, in this paper, by a `smooth (resp. topological) orbifold' we
  mean a \emph{smooth (resp.  topological) stack}, or more precisely,
  a \emph{stack in the category $\mathsf{Diff}$
    (resp. $\mathsf{Top}$)}.  (In our exposition we have attempted to
  make the language of stacks accessible to a non-expert audience
 (more on this below), although in no way do we aim to be a textbook.
 We suggest \cite{BCEFFGK-stacks, Lerman:2010,
    Met03} for more about stacks from a topologist's point of view; a
  complete beginner may wish to start with \cite{Edidin:2003} or
  \cite{Fantechi:2001}.)  In particular, the notion of `equivalence'
  in the previous paragraph is a (weak) equivalence of the underlying
  categories of the stacks.  It is important to note that such an
  equivalence is more restrictive than a homeomorphism of underlying
  topological spaces; this is because such a homeomorphism does not
  retain any `stacky' information relating to the orbifold
  singularities.  A basic class of examples illustrating this
  distinction are the so-called \emph{weighted projective stacks},
  i.e. 
  $\P(a_0, \ldots, a_n) = [(\C^{n+1}\setminus \{0 \})/\C^*]$, where
  $\C^*$ acts diagonally on $\C^{n+1} \setminus \{0 \}$ with weights
  $a_0, \ldots, a_n \in \Z_{+}$. (Here we follow the convention in the
  literature and denote by $[X/G]$ the \emph{stack} associated to a
  $G$-action on a space $X$; the topological quotient (orbit) \emph{space} is
  denoted $X/G$.) Although the underlying orbit space
  $(\C^{n+1}\setminus \{0 \})/\C^*$ is homeomorphic to the orbit space
  of a finite group action on complex projective space $\C P^n$ (see
  \cite{Kaw73} for details), the stack $\P(a_0, \ldots, a_n) $ is not
  equivalent \emph{as a stack} to a global quotient (except in the
  trivial case when all of the weights are equal to $1$ and $\P(1,
  \ldots, 1) = \C P^n$ is a smooth manifold).  This follows from
  Theorem \ref{thm:TFAE} below, but is also well-known---see
  e.g. \cite{AdemLeidaRuan:2007}.

Secondly, we emphasize that we restrict ourselves throughout this paper to orbifolds arising as quotient stacks
$[X/G]$, where $G$ is a Lie group acting smoothly and properly on a
smooth manifold $X$. It is worth noting that 
all \emph{reduced}, or \emph{effective} orbifolds -- 
orbifolds whose local isotropy groups act effectively -- are known to be
of this type by a frame-bundle construction (see, for example,
\cite{AdemLeidaRuan:2007}), so this is not a very restrictive
condition in practice. 
Moreover, as we already mentioned, the main application we have in mind of our 
Theorem~\ref{thm:TFAE} is to the beautiful class of
quotient stacks 
known as toric Deligne-Mumford stacks, which are stack analogues of
smooth toric varieties.  

Thirdly, we recall that in his foundational work \cite{Noohi:2005}, Noohi deals with topological stacks and the theory of (stacky) fundamental groups and coverings in a very general framework. We owe much to \cite{Noohi:2005} for both the mathematical content and the exposition in Sections~\ref{sec:preliminaries},~\ref{section:pi1 of quotient stacks}, and ~\ref{section:quotients and covers}. Indeed, our Theorem~\ref{thm:TFAE}, quoted below, can be viewed as an extension of \cite[Theorem 18.24]{Noohi:2005} in the special case of quotient stacks. By restricting to quotient stacks, we are able to connect aspects of stacky algebraic topology with a distinctly geometric (and classical) condition on a Lie group acting on a manifold; in particular, our proofs are our own.  In our exposition we have attempted to explicitly preserve the classical perspective and language as much as possible; we hope this serves to illustrate to a broad audience the appeal of the stack perspective, and to further elucidate the insights of \cite{Noohi:2005} in a setting that is common in differential geometry, namely Lie groups acting on manifolds.

With these points in mind we now state our main general result on stacks arising as global quotients (Theorem
\ref{thm:TFAE} in Section~\ref{section:quotients and covers}). 
We refer to Section \ref{subsec:fundamental group}  for the definition of the inertia homomorphism in (\ref{item:inertia}), and Section \ref{section:quotients and covers} for a discussion of (stacky) covering projections appearing in (\ref{item:cover}).  For now, 
the reader may keep in mind that the \emph{inertia groups} $I_x$ mentioned below can be identified with 
isotropy groups $\Stab(p) \subset G$ of certain points $p\in X$.   

\begin{reptheorem}{thm:TFAE} 
Let $X$ be a simply connected manifold, equipped with a smooth proper action of a Lie group $G$.  Let $G_0\subset G$ denote the connected component of the identity element of $G$, and $I_x$ the inertia group of $x \in [X/G]$.  The following statements are equivalent.
\begin{enumerate}
\item $[X/G]$ is equivalent to a discrete quotient. \label{item:quotient}
\item $G_0$ acts freely on $X$. \label{item:acts freely}
\item For all $x$ in $X$, the inertia homomorphism $\omega_x:I_x \to \pi_1([X/G],x)$ is injective. \label{item:inertia}
\item The (stacky) universal cover of $[X/G]$ is equivalent to a smooth manifold. \label{item:cover}
\end{enumerate}
\end{reptheorem}

Though conditions (\ref{item:quotient}), (\ref{item:inertia}), and (\ref{item:cover}) are known to be equivalent by \cite[Theorem 18.24]{Noohi:2005}, we provide a different proof of these equivalencies for the case of quotient stacks by showing each of these conditions is in turn equivalent to (\ref{item:acts freely}). 
The distinctly geometric appeal of condition (\ref{item:acts freely}),
in comparison to the topological nature of conditions
(\ref{item:inertia}) and (\ref{item:cover}), also carries over in our
interpretation of this result in the setting of smooth toric DM stacks.

We now take a moment to briefly recall the context of our discussion of \emph{toric Deligne-Mumford (DM) stacks} 
in Section~\ref{sec:toricDMstacks}. 
In their foundational paper \cite{BCS05},
Borisov, Chen, and Smith introduce the notion of a \emph{stacky fan},
the combinatorial data from which one constructs a toric
  Deligne-Mumford (DM) stack using an anaolgue of the Cox quotient
construction in algebraic geometry.  
In \cite{FantechiMannNironi:2010}, an abstract
definition of a smooth toric DM stack was given, which was shown to be
compatible with the construction of Borisov, Chen, and Smith \cite{BCS05}
(see also \cite{Iwanari:2009b},
\cite{GeraschenkoSatriano:2011a,GeraschenkoSatriano:2011b} for related
approaches). 
From the symplectic geometry perspective, Lerman and Malkin
\cite{LermanMalkin2009} gave a definition of symplectic toric DM
stacks (in the smooth category $\mathsf{Diff}$), offering a modern perspective on symplectic toric orbifolds.
In fact, parallel to the classical theory
of toric varieties, there is a subfamily of toric DM stacks---namely
those toric DM stacks whose underlying fan is polytopal---that admit
a construction from symplectic geometry via \emph{stacky
  polytopes}, using an adaptation of the construction of Borisov,
Chen, and Smith (see \cite{Sakai2010}). In particular, from this
construction, this subfamily 
can be seen to give examples of symplectic toric DM stacks. 
(In earlier work, Lerman and Tolman extended the Delzant
classification of (compact) symplectic toric manifolds to symplectic
toric orbifolds; in the orbifold case, the classification is by `labelled
polytopes' --- i.e. polytopes with positive integer labels attached to each
facet \cite{LT97}.) 
In this manuscript, 
we interpret our analysis of
orbifolds as global quotients in this class of examples; our results 
are explicit and combinatorial, stated in terms of the stacky fan/polytope.

In order to state our main results for smooth toric DM stacks, recall
that a stacky fan is a triple $(N,\Sigma,\beta)$ consisting of a
finitely generated $\Z$-module $N$, a simplicial fan $\Sigma$ in
$N\otimes \R$ with $n$ rays $\rho_1, \ldots, \rho_n$, and a
homomorphism $\beta:\Z^n \to N$ satisfying certain conditions (see
Definition \ref{definition:stacky fan}). By interpreting Theorem
\ref{thm:TFAE} in this case, we can characterize global quotients
among toric DM stacks in terms of their underlying stacky fans. Let
$N'$ denote the image of $\beta$, and for a cone $\sigma$ in $\Sigma$,
let $N_\sigma$ denote $ \mathrm{span} \{ \beta(e_i) \, | \, \rho_i
\text{ is a ray in } \sigma \}$, where $e_i$ denotes the $i$-th
standard basis vector in $\Z^n$. In this context, condition
(\ref{item:acts freely}) of Theorem \ref{thm:TFAE} results in the
following 
Corollary, which characterizes the stacky fans yielding global
quotients. 
We note that the equivalence of 
condition
(\ref{item:acts freely}) of Theorem \ref{thm:TFAE} 
and the combinatorial condition stated in Corollary~\ref{cor:globalquotient}, in the context of toric DM stacks, 
was first proved via a different method -- a
combinatorial 
analysis of the local isotropy groups -- in joint work of the authors
with Goldin and Johanssen; this approach is described in 
\cite{GHJK:2011}.
We also note that in \cite{PoddarSarkar:2010}, the authors study
quasi-toric orbifolds using techniques from toric topology, and obtain
similar results describing universal covers of quasi-toric orbifolds in that framework.

\begin{repcorollary}{cor:globalquotient}
  Let $(N,\Sigma,\beta)$ be a stacky fan, and let $\mathcal{X}$ denote
  the corresponding toric DM stack. Then $\mathcal{X}$ is equivalent
  to a global quotient if and only if $N'=N_\sigma$ for all maximal
  cones $\sigma$ in $\Sigma$.
\end{repcorollary}

In addition to Corollary~\ref{cor:globalquotient}, we interpret the other obstructions appearing in
conditions (\ref{item:inertia}) and (\ref{item:cover}) of Theorem
\ref{thm:TFAE} for toric DM stacks in terms of stacky fans. 
Specifically, using results in \cite{GHJK:2011}, we see
that the inertia homomorphisms of condition (\ref{item:inertia}) can
be identified with very natural homomorphisms defined in terms of the
stacky fan (Proposition \ref{prop:inertia homomorphism for DM
  stacks}).  
  Additionally, we describe the universal cover of a toric
DM stack in terms of its stacky fan in Proposition \ref{prop:universal
  cover of toric stack}.
It is interesting to note that the universal cover of a toric DM stack
is given in terms of its stacky fan data, and is hence also a toric DM
stack arising from a stacky fan. In particular, global quotients among
toric DM stacks are quotients of toric \emph{manifolds} by finite
group actions. We expect such a nice description to be useful in
computations of orbifold/stacky invariants, particularly for global
quotients (cf. \cite{AdeRua03}, \cite{FG03}, for example).



\bigskip

The contents of this paper are as follows.  After a brief discussion
of stacks and fundamental groups of stacks in Section
\ref{sec:preliminaries}, we specialize to quotient stacks in Section
\ref{section:pi1 of quotient stacks}, where we describe in detail the
fundamental group, and inertia homomorphism of quotient stacks.  In
Section \ref{section:quotients and covers}, we describe the universal
cover of a quotient stack and prove Theorem \ref{thm:TFAE} stated
above; analogous results for connected but non-simply connected $X$
are also explored in this section.  In Section
\ref{sec:toricDMstacks}, we turn our attention to toric DM stacks,
where we describe the inertia homomorphism (Section
\ref{subsec:toricinertia}) and universal cover (Section
\ref{subsec:toriccover}) of toric DM stacks.  We also verify an
expected relation between the symplectic volumes (in the stacky
polytope case) of a symplectic toric DM stack and its universal cover,
and the corresponding volumes of the underlying polytopes.  We
conclude with some examples in Section \ref{subsec:examples}.

\medskip
\noindent \textbf{Acknowledgements.} We thank Anthony Bahri for asking
us the question which was the catalyst for this project, and Andrew Nicas for
continuing support and many useful conversations.

\section{Preliminaries}\label{sec:preliminaries}

This section establishes notation and collects some facts about 
stacks.  We mainly follow the notation of  \cite{Noohi:2005} and \cite{BCEFFGK-stacks}.

\subsection{Stacks}

We will mainly work over the base category $\mathsf{Diff}$ (smooth
manifolds and smooth maps), although occasionally we shall work over 
$\mathsf{Top}$ (topological spaces and continuous maps). To streamline
some of the discussion below, we use $\mathsf{Sp}$ to denote either $\mathsf{Diff}$
or $\mathsf{Top}$, and we refer to the corresponding objects
simply as spaces and the morphisms as maps.

For a space $M$, let $\und{M}$ denote its associated stack, with objects
$\{f:E\to M\}$, i.e. the collection of maps in $\mathsf{Sp}$ with
target $M$, and with 
morphisms $\{g:E \to E' \, | \, f' \circ g = f\}$. Given a   map
$F:M\to M'$, we write $\und{F}:\und{M} \to \und{M}'$ for the induced
morphism of stacks.
Fix a terminal object $\star$ in $\mathsf{Sp}$. 
For a choice of point $p$ in a space $M$, let $\und{p}$ be the corresponding
point in
$\und{M}$ (i.e. by abuse of notation, $p$ also denotes the map $\star
\mapsto p
\in M$ and $\und{p}$ the induced morphism of stacks).  More generally, recall
that a
point $x$ in a stack $\mathcal{X}$ is a morphism $x:\und{\star} \to
\mathcal{X}$.

A morphism $F:\mathcal{Y} \to \mathcal{X}$ of stacks is
\emph{representable} if for every morphism  $\und{W} \to \mathcal{X}$ whose
source is (the stack associated to) a space, the fiber product $\mathcal{Y}
\times_{\mathcal{X}} \und{W}$ is equivalent to (the stack associated to) a
space.  In this case, we shall call the induced map $\mathcal{Y}
\times_{\mathcal{X}} \und{W} \to \und{W}$ a representative of $F$. Notice that
the representative of a representable map is (induced by) a map of spaces.

\begin{definition} \cite[Definition 66]{Met03}
A stack $\mathcal{X}$ over $\mathsf{Sp}$ is \emph{locally representable} if
there exists a representable epimorphism of stacks $p:\und{X} \to \mathcal{X}$.
In this case, $p$ is called a \emph{presentation} of $\mathcal{X}$.
\end{definition}

\begin{remark}
In  \cite{Noohi:2005}, a locally representable stack over
$\mathsf{Top}$ is called a \emph{pretopological stack}.
\end{remark}

Many properties of representable morphisms of stacks are defined by the
corresponding properties of their representatives, which are maps of spaces.
For example we have the following (cf. \cite{Noohi:2005}). 

\begin{definition}
 A morphism $\mathcal{Y}\to \mathcal{X}$ of locally representable stacks in
$\mathsf{Top}$ is a \emph{covering projection} if it is representable and if
every representative is a covering projection.
\end{definition}

Several other properties of maps in $\mathsf{Sp}$ can often be
translated similarly into properties of representable maps of stacks (e.g. proper, closed, submersion, etc. see \cite[Section 3.3]{Met03}).  For
now, we simply recall that an important necessary condition for a
property of maps to be thus translatable is that 
they be stable under taking pullbacks.  (If
$\mathsf{Sp}=\mathsf{Diff}$, it is additionally required that the property be
stable under pullbacks via submersions.) We record the following
definition from \cite{Met03}, which connects our point of view with
that of orbifolds and orbifold charts.

\begin{definition} 
A locally representable stack
  $\mathcal{X}$ over $\mathsf{Diff}$ is called an \emph{orbifold} or a \emph{Deligne-Mumford stack}  if  it admits a presentation $p:\und{X} \to \mathcal{X}$ that is \'etale
  and if the diagonal map $\Delta:\mathcal{X} \to \mathcal{X} \times
  \mathcal{X}$ is proper.
\end{definition}

In the literature, the term orbifold is at times reserved for Deligne-Mumford stacks with trivial global (generic) stabilizer (inertia group), what is elsewhere called an \emph{effective} or \emph{reduced} orbifold.  The above definition allows for non-trivial global stabilizer.

An important instance of the above definition is the following.
 Let $X$ be a manifold equipped with a smooth action by a compact Lie group
$G$. If the action is locally free (i.e. with finite isotropy groups), then the
associated quotient stack $[X/G]$ (see Section \ref{section:pi1 of quotient
stacks}) is an orbifold. This will be our main source of examples. 

We shall at times wish to view a \emph{geometric} stack over
$\mathsf{Diff}$---a locally representable stack whose presentation is a surjective submersion---as a stack over $\mathsf{Top}$.  To say this more precisely, recall that given a geometric stack with presentation $\und{X_0} \to \mathcal{X},$ there is a Lie groupoid $\mathcal{G}= (X_1 \rightrightarrows X_0)$ whose associated stack $B\mathcal{G}$ (see \cite[Section 4]{Lerman:2010}) is equivalent to $\mathcal{X}$.  (In this case, we say that $\mathcal{G}$ is a \emph{groupoid presentation} for $\mathcal{X}$.) Considering the Lie groupoid $\mathcal{G}$ as a groupoid object in $\mathsf{Top}$, we view $B\mathcal{G}$, and hence $\mathcal{X}$, as a stack over $\mathsf{Top}$.

\subsection{Fundamental group of topological
  stacks}\label{subsec:fundamental group}

Following the work of Noohi \cite{Noohi:2005}, 
we now recall some of the main definitions surrounding the fundamental group of a topological stack that we later 
interpret more explicitly for quotient stacks $\mathcal{X} =
[X/G]$ as above.  


In this subsection, we work over
$\mathsf{Top}$. 
Let $\mathcal{X}$
and $\mathcal{Y}$ be stacks and $x,y$ points in
$\mathcal{X},
\mathcal{Y}$ respectively.  Recall that a
pointed map $(f,\alpha):  (\mathcal{Y},y) \to (\mathcal{X}, x)$ of
stacks consists of a morphism $f:\mathcal{Y} \to \mathcal{X}$ and a
2-isomorphism $\alpha: x\leadsto f (y)=f\circ y$, where 
as in
\cite{Noohi:2005} 
we use the symbol $\leadsto$ for a 2-isomorphism of points, reserving the symbol $\Rightarrow$ for other 2-isomorphisms.
If $(f,\alpha)$ and $(g,\beta)$ are
pointed maps $(\mathcal{Y},y) \to (\mathcal{X},x)$, a
(pointed) 2-isomorphism
$\epsilon: (f,\alpha) \Rightarrow (g,\beta)$ is a $2$-isomorphism
$\epsilon:f\Rightarrow g$ such that the induced 2-isomorphism
$\epsilon(y)~:~f(y) \leadsto g(y)$ satisfies
$\epsilon(y)\alpha=\beta$.  (Here $\epsilon(y)= \epsilon *
\mathrm{id}_y$, the `horizontal' composition of $2$-morphisms, while
$\epsilon(y)\alpha$ denotes the `vertical' composition of
$2$-morphisms as in \cite{Lerman:2010}.)

 Let $M$ be a topological space with a chosen point $p$ and $(\und{M},\und{p})$
 its associated pointed stack. We begin with a remark regarding 
morphisms (resp. pointed morphisms) from $\und{M}$
(resp. $(\und{M}, \und{p})$) to a stack $\mathcal{X}$
(resp. $(\mathcal{X}, x)$). 

\begin{remark} \label{remark:howtorepresentmaps} 
Let $M$ be an object in $\mathsf{Top}$. A morphism $f:\und{M}
  \to \mathcal{X}$ 
determines an
  object $E_f$ in $\mathcal{X}$ (over $M$) given by evaluation at the
  identity $E_f:=f( \mathrm{id}_M:M\to M)$.  In fact, evaluation at
  the identity defines a functor $\mathrm{ev}_M$ from
  $\mathrm{HOM}(\und{M},\mathcal{X})$ to $\mathcal{X}_M$ which is
  surjective on objects and fully faithful (Proposition 2.20 in
  \cite{BCEFFGK-stacks}).  That is, given an object $E$ in
  $\mathcal{X}$ over $M$ there is a morphism $g:\und{M} \to
  \mathcal{X}$ such that $E_g=E$. Moreover, given an isomorphism
  $\varphi:E_g\to E_f$ in $\mathcal{X}_M$, there exists a unique
  $2$-isomorphism, which we will also denote $\varphi:g\Rightarrow f$,
  whose evaluation at the identity is $\varphi$. 
  Thus, up to canonical $2$-isomorphism a map $\underline{M} \to \calX$ is determined by the  data of an object $E$ in $\calX$ over $M$. 

  A similar reasoning applies to pointed maps.  Fix points $\und{p}$
  in $\und{M}$ and $x$ in $\mathcal{X}$.  Let
  $E_x=x(\mathrm{id}_{\star}: \star \to \star)$ in
  $\mathcal{X}_\star$.  A pointed
  map $(f,\alpha): (\und{M}, \und{p}) \to (\mathcal{X},x)$ determines
  objects $E_f$ and $E_{f(p)}$ (over $M$ and $\star$, respectively), a
  morphism $E_{f(p)} \to E_f$ (over $p:\star \to M$) and an
  isomorphism $\alpha: E_x \to E_{f(p)}$.  Conversely, as in the proof
  of Proposition 2.20 in \cite{BCEFFGK-stacks}, given objects $E$ in
  $\mathcal{X}_M$ and $E_0$ in $\mathcal{X}_\star$ and a morphism $E_0
  \to E$ over $p:\star \to M$, there is a morphism $f:\und{M} \to
  \mathcal{X}$ such that $E_f=E$, $E_{f(p)}=E_0$ and the induced
  morphism $E_{f(p)} \to E_f$ is the given one.  Since
  $\mathrm{ev}_\star$ is fully faithful, a choice of isomorphism
  $\alpha: E_x \to E_0$ then determines a unique $2$-isomorphism,
  which we will also denote $\alpha:x \leadsto f(p)$, and hence a
  pointed map $(f,\alpha)$ whose evaluation at the identity yields the
  data $E_0 \to E$ with the isomorphism $\alpha$. 

 Moreover, given pointed maps $(f,\alpha)$ and $(g,\beta)$ and isomorphisms $\varphi:E_g \to E_f$ and $\psi:E_{g(p)} \to E_{f(p)}$ such that the diagram
 $$
 \xymatrix@R=1em{
 & E_{g(p)} \ar[r] \ar[dd]^{\psi}& E_g \ar[dd]^{\varphi} \\
 E_x \ar[ur]^\alpha \ar[dr]_\beta & & \\
 & E_{f(p)} \ar[r] & E_f
 }
 $$
 commutes, there is a unique 2-isomorphism $\epsilon:(g,\beta) \Rightarrow (f,\alpha)$
  with $\mathrm{ev}_M(\epsilon)=\varphi$ and $\mathrm{ev}_\star(\epsilon(p) )=  \psi$.
   As in the unpointed case above, we  conclude that a pointed map $(\und{M}, \und{p}) \to (\calX, x)$ is determined up to unique $2$-isomorphism by objects $E$ in $\calX_M$ and $E_0$ in $\calX_\star$, a morphism $E_0 \to E$ over $p: \star  \to \und{M}$, and a choice of $2$-isomorphism $\alpha: E_0 \to E_x$. 
 \end{remark}

 Let $I$ denote the unit interval
 $[0,1]$.  Recall that a {\em homotopy} of pointed maps $(f,\alpha),(g,\beta):
 (\und{M}, \und{p}) \to (\mathcal{X},x)$ consists of: a morphism $H:
 \und{I\times M} \cong \und{I} \times \und{M} \to \mathcal{X}$, a $2$-isomorphism for the
 diagram
$$
\xymatrix{ {\und{I \times \star} \cong \und{I}} \ar[r]  \drtwocell<\omit>{<0>}
               \ar[d]_{{\und{i}_1}} 
                                   & {\und{\star}} \ar[d]^{x}    \\
                                        {\und{I} \times \und{M}} \ar[r]_{H}        & {\mathcal{X}}           }
$$
and a pair
of pointed 2-isomorphisms $\epsilon_0:f\Rightarrow H_0$ and $\epsilon_1: H_1
\Rightarrow g$.  Here $i_1= (\mathrm{id}, p): I \to I\times M$ is the
inclusion into the first factor, and $H_0$ and $H_1$ denote the
pointed maps induced by restriction to $\und{\{0\}} \times \und{M}$
and $\und{\{1\}} \times \und{M}$ respectively. When such a homotopy
exists, we shall sometimes say that $(f,\alpha)$ is homotopic to
$(g,\beta)$ or that $(H,\epsilon_0,\epsilon_1)$ is a homotopy from
$(f,\alpha)$ to $(g,\beta)$.  This induces an equivalence relation on
the set of pointed maps $(\und{M}, \und{p}) \to (\calX,x)$ and we
denote by $[(\und{M},\und{p}),(\mathcal{X},x)]$ the resulting set of
equivalence classes (i.e. the set of homotopy classes of pointed
maps).

That the relation above is reflexive and symmetric is easily verified.  To show
it is also transitive, we must be able to `compose' homotopies---that is, given
a homotopy $(H,\epsilon_0,\epsilon_1)$ from $(f,\alpha)$ to $(g,\beta)$ and a
homotopy $(H',\epsilon'_0, \epsilon'_1)$ from $(g,\beta)$ to $(h,\gamma)$,
construct a homotopy $(H'', \epsilon''_0, \epsilon''_1)$ from $(f,\alpha)$ to
$(h,\gamma)$.  This is proved in Lemma 17.4 in \cite{Noohi:2005} in greater
generality.  (In the case of pointed maps whose domain is the associated stack of
a space $M$, we may observe the composition of homotopies more readily in light
of Remark \ref{remark:howtorepresentmaps}.  Indeed, the objects $E_H$ and
$E_{H'}$ in $\calX_{I\times M}$ can be glued together to form the required
object $E_{H''}$ that yields the desired composition of homotopies.)

As in \cite[Remark 17.3]{Noohi:2005}), observe that a 2-isomorphism $\epsilon:(f,\alpha) \Rightarrow (g,\beta)$ of pointed maps   induces a homotopy.  Indeed, let $H$ be the composition $\und{I} \times \und{M} \stackrel{\und{\mathrm{pr}_2}}{\longrightarrow} \und{M} \stackrel{f}{\longrightarrow} \mathcal{X}$, with $2$-isomorphism as indicated by the diagram
$$
\xymatrix{ {\und{I}} \ar[rr]  \drtwocell<\omit>{<0>\,{\mathrm{id}}}
               \ar[d]_{{\und{i}_1}} 
                                   & & {\und{\star}} \ar[dl]_{\und{p}} \ar[dr]^{x}\drtwocell<\omit>{<2.5>\,{\alpha}} &   \\
                                        {\und{I} \times \und{M}}  \ar[r]_{\mathrm{pr}_2}    & {\und{M}} \ar[rr]_{f}&        & \mathcal{X}
}
$$
(i.e. the $2$-isomorphism $(\mathrm{id}_f * \mathrm{id})(\alpha * \mathrm{id}_\star)$), $\epsilon_0=\mathrm{id}$ and $\epsilon_1=\epsilon$.  This observation proves the following Lemma.

\begin{lemma}\label{lemma:homotopy class independent of 2-morphism}
  Let $p$ be a point in a space $M$ and let $(\calX, x)$ be a pointed stack. Let
$(f,\alpha)$ and $(g,\beta)$ be pointed maps $(\und{M}, \und{p}) \to (\calX,x)$.
If there exists a
 2-isomorphism of pointed maps $(f,\alpha) \Rightarrow (g,\beta)$, then the
  homotopy classes of $(f,\alpha)$ and $(g,\beta)$ are equal.
\end{lemma}

\begin{remark} \label{remark:representingmaps}
 It follows from Remark \ref{remark:howtorepresentmaps} that the homotopy class of a pointed map $(\und{M},\und{p}) \to (\mathcal{X},x)$ is determined by specifying objects $E$ in $\mathcal{X}_M$ and $E_0$ in $\mathcal{X}_\star$, along with a morphism $E_0 \to E$ over $p:\star \to M$ and an isomorphism $\alpha:E_x \to E_0$.  
\end{remark}

Let  $1\in S^1$, the unit circle in $\C$.

\begin{definition}
  Let $x$ be a point in the stack $\mathcal{X}$.  Define the {\em
    fundamental group of $\mathcal{X}$} to be the set of homotopy
  classes of pointed maps $(\und{S}^1,\und{1}) \to (\mathcal{X},x)$,
  denoted
  $\pi_1(\mathcal{X},x):=[(\und{S}^1,\und{1}),(\mathcal{X},x)]$.
  \end{definition}

The reader may wish to consult \cite{Noohi:2005} for details concerning the group structure of $\pi_1(\mathcal{X},x)$.  


In Section \ref{section:quotients and covers}, we shall consider covering projections of (connected) quotient  stacks.  Making use of \cite[Corollary 18.20]{Noohi:2005}, we say that a covering projection $\mathcal{Y} \to \mathcal{X}$ is the \emph{universal covering projection} if $\pi_1(\mathcal{Y},y)$ is trivial.

 An interesting feature of the fundamental group of a stack is the following natural homomorphism $\omega_x: I_x \to \pi_1(\mathcal{X},x)$, where $I_x := \{ \alpha: x \leadsto x\}$ is the {\em inertia group of $\mathcal{X}$ at $x$}. 
 The homomorphism $\omega_x$ is defined as follows (cf. \cite[Remark 17.3]{Noohi:2005}). Given $\alpha \in I_x$, let $H_\alpha:\und{I} \to \mathcal{X}$ be defined by the composition $\und{I} \to \und{\star} \stackrel{x}{\longrightarrow} \mathcal{X}$ and consider the pair of 2-isomorphisms $\epsilon_0=\mathrm{id}$ and $\epsilon_1=\alpha$ in the following diagrams.
$$
\xymatrix{
 & \und{\star} \ar[dl]_{\iota_0}
 \ar[dr]^{x}\drtwocell<\omit>{<2.5>\,{\epsilon_0}} & \\
 \und{I} \ar[rr]_{H_\alpha} & & \mathcal{X}
} \quad \quad
\xymatrix{
 & \und{\star} \ar[dl]_{\iota_1} \ar[dr]^{x} \drtwocell<\omit>{<2.5>\,{\epsilon_1}}& \\
 \und{I} \ar[rr]_{H_\alpha} & & \mathcal{X}
}
$$
This data glues together to give a pointed map $(\und{S}^1, \und{1})
\to (\mathcal{X},x)$  
whose homotopy class is denoted $\omega_x(\alpha)$.  

\section{The fundamental group of $\mathcal{X}=[X/G]$} \label{section:pi1 of
quotient stacks}

We now specialize to our case of interest, stacks $\mathcal{X} =
[X/G]$, where $X$ is a smooth manifold equipped with a (right) action
of a Lie group $G$.  In this case, recall that the objects of
$\mathcal{X}$ are pairs of maps $(B\leftarrow E \to X)$ consisting of a (right)
principal $G$-bundle $E\to B$ and a $G$-equivariant map $E\to X$ and
that the morphisms of $\mathcal{X}$ are pairs $(E\to E', B\to B')$
such that in the following diagram
$$
\xymatrix@C=1.5em@R=1.5em{
 & & X \\
 E \ar@/^/[urr] \ar[r] \ar[d] & E' \ar[d] \ar@/_/[ur] & \\
 B \ar[r] & B' & \\
}
$$
the square is Cartesian and the triangle commutes.  Fix a point $x$ in
$\mathcal{X}$, and write $( \star \leftarrow  E_x \to X)$ to denote the
object $x(\mathrm{id}_\star)$ in $\mathcal{X}_\star$.  By choosing a
base point $e_0 \in E_x$ we shall further identify $E_x \cong G$, in
which case the $G$-equivariant map $x:E_x \to X$ is realized by the
map $x_p:G\to X$, $g\mapsto p\cdot g$ where $p:=x(e_0)$. Thus $x_p$ parameterizes the
$G$-orbit through $p$.

An element $\alpha \in I_x$ induces a gauge transformation $\alpha: E_x \to E_x$ such that the diagram below commutes.
$$
\xymatrix@C=1.5em@R=1.5em{ & & X \\
E_x \ar[r]^{\alpha} \ar@/^/[rru]^x & E_x \ar@/_/[ru]_x 
}
$$
 Moreover, any such gauge transformation determines an element in $I_x$.   Under the identification $E_x \cong G$ determined by $e_0 \leftrightarrow 1$, any $\alpha \in I_x$ is determined by $g_\alpha = \alpha(1)$, which must lie in $\Stab(p)$ since the above diagram commutes. That is, $\alpha:G\to G$ is simply left multiplication by $g_\alpha$.  In what follows, we will implicitly use the isomorphism $\alpha \mapsto g_\alpha$ to identify $I_x \cong \Stab(p)$.

In this setting, the homomorphisms $\omega_x$ discussed at the end of 
Section~\ref{subsec:fundamental group} may be described explicitly.
The data determined by $H_\alpha:\und{I} \to \mathcal{X}$ yields the
trivial bundle $I \times G \to I$ with $G$ equivariant map $I\times G
\to X$ given by $(t,g)\mapsto p\cdot g$.  The 2-isomorphisms
$\epsilon_0 =\mathrm{id}$ and $\epsilon_1=\alpha$ yield the
identifications of fibers 
\begin{align*}
\{0\} \times G \to G & & G\to \{1\} \times G \\
(0,g) \mapsto g & & g \mapsto (1,g_\alpha g).
\end{align*}
This assembles to  give a bundle over $S^1$ and an equivariant map
representing $\omega_x(\alpha)$ (cf. Remark
\ref{remark:representingmaps}). Namely, we take 
\begin{equation}\label{eq:definition E alpha}
E({g_\alpha}) :=( I \times G )/\sim, \text{ where the equivalence relation
  $\sim$ is }  (0,g) \sim (1,g_\alpha g) \textup { for } g \in G,
\end{equation}
and the equivariant map is given by $E({g_\alpha}) \to X$, $[(t,g)] \mapsto p \cdot
g$. 

\medskip

Let $BG:=[\star/G]$ and let $*:\star \to BG$ be the choice
of point whose value at $S\to \star$ is the trivial bundle $S
\times G \to S$.

\begin{lemma}\label{lemma:fundamental group of BG} (Cf. \cite[Example 4.2]{Noohi:2004}) 
Let $G$ be a Lie group and let $G_0$ denote the  component of the identity element in $G$.  Then $\pi_1(BG,*) \cong G/G_0$.
\end{lemma}

\begin{proof}
We define a map $\sigma: G/G_0 \to \pi_1(BG, *)$ and show that $\sigma$ is an isomorphism.  For $g\in G$ let $E(g)\to S^1$
be the $G$-bundle  $E(g) = (I\times G)/ (0,h) \sim (1,gh)$.  Let
$\sigma(gG_0)=[(f,\alpha)]$ where $(f,\alpha)$ is a morphism
determined by the bundles $E(g)\to S^1$ and $E(g)|_{\{1\}}\to \star$,
the inclusion of the fiber $E(g)|_{\{1\}}\to E(g)$ over $1:\star \to
S^1$, and the $2$-isomorphism $\alpha$ determined by the
identification  $G \to  E(g)|_{\{1\}}$, $h\mapsto [(0,h)]$. Here the
brackets $[ \hspace{1mm} ]$ denotes equivalence classes with respect
to the relation in~\eqref{eq:definition E alpha}. (Recall that by Remark \ref{remark:representingmaps}, any other pointed map determined by this data differs from $(f,\alpha)$ by a unique 2-isomorphism of pointed maps, which by Lemma \ref{lemma:homotopy class independent of 2-morphism} lies in the same homotopy class.)  To see that $\sigma$ is well-defined, observe that if $\gamma:I\to G$ is any path from $g$ to $g'$, we may construct a bundle isomorphism $\varphi_\gamma:E(g) \to E(g')$ by the formula $[(t,h)] \mapsto [(t,\gamma(t)g^{-1}h)]$ that fits in the diagram 
\begin{equation} \label{diagram:pointedisomorphism}
 \xymatrix@R=1em{
 & E({g})|_{\{1\}} \ar[r] \ar[dd]^{\varphi|_{\{1\}}}& E(g) \ar[dd]^{\varphi} \\
 G \ar[ur]^\alpha \ar[dr]_{\alpha'} & & \\
 & E(g')|_{\{1\}} \ar[r] & E(g')
 }
\end{equation}
where $\alpha$ and $\alpha'$ are the resulting isomorphisms over $\star$.  Let us call such a bundle isomorphism (i.e. one which respects the given trivializations over $\star$) a \emph{pointed isomorphism}. By Lemma \ref{lemma:homotopy class independent of 2-morphism}, we see that the map $\sigma$ is well-defined.  

Since every $G$-bundle over $S^1$ is isomorphic to $E(g)$ for some $g$, it is easy to see that $\sigma$ is surjective. It remains to show that $\sigma$ is injective.  To that end, suppose that $[(f,\alpha)]=\sigma(gG_0)=\sigma(g'G_0)=[(f'\alpha')]$.  Let $(H,\epsilon_0,\epsilon_1)$ be a homotopy from $(f,\alpha)$ to $(f',\alpha')$.  We claim that this results in a pointed isomorphism of bundles as in the diagram (\ref{diagram:pointedisomorphism}) above, which will complete the proof.  Indeed, write the resulting bundle isomorphism  $E(g) \to E(g')$ as $[(t,h)] \mapsto [(t,\phi(t)h)]$, where $\phi:I\to G$ is continuous and satisfies $g'\phi(0) = \phi(1) g$.  That the isomorphism is pointed (i.e. must fit in the diagram (\ref{diagram:pointedisomorphism})) forces $\phi(0)=1$ and we may construct a path $\gamma(t) = \phi(t)g$ joining $g$ and $g'$, whence $gG_0 = g'G_0$.  

We shall now verify the claim that the homotopy produces the required pointed isomorphism of bundles $E(g) \to E(g')$.
Let $E \to I \times S^1$ be the bundle given by $H(\mathrm{id}:I\times S^1 \to I \times S^1)$ and let $E_j \to S^1$ ($j=0,1$) denote the bundles given by $H(\varepsilon_j:S^1=\{j\} \times S^1 \hookrightarrow I \times S^1)$.   Note that the 2-isomorphisms of pointed maps $\epsilon_j$ yield pointed isomorphisms of bundles $E(g) \to E_0$ and $E_1 \to E(g')$.  Therefore, it suffices to find a pointed isomorphism $E_0 \to E_1$. 

Let $\eta$ denote the $2$-isomorphism  for the diagram 
\[
\xymatrix{
 \und{I} \ar[r]  \drtwocell<\omit>{<0>\,{\eta}}
               \ar[d]_{{\und{i}_1}} 
                                   & {\und{\star}} \ar[d]    \\
                                        {\und{I} \times \und{S}^1} \ar[r]_{H}        & {BG}           }
\]
that in turn induces a trivialization $\eta:I \times G \to E_I$, where $E_I$ denotes the bundle over $I$ corresponding to $H(i_1:I \to I \times S^1)$.  We therefore seek a pointed isomorphism
$$
 \xymatrix@R=1em{
 & E_{0}|_{\{1\}} \ar[r] \ar[dd]^{\varphi|_{\{1\}}}& E_0 \ar[dd]^{\varphi} \\
 G \ar[ur]^{\eta_0} \ar[dr]_{\eta_1} & & \\
 & E_{1}|_{\{1\}} \ar[r] & E_{1}
 }
$$
where the identifications $\eta_j$ of the fibers over $1\in S^1$ are induced from the trivialization $\eta$.
By Theorem 9.8 in \cite[Ch. 4]{Husemoller:1994}, there is a bundle isomorphism $\psi: E \to I \times E'$ where $E'$ is a $G$-bundle over $S^1$, inducing an isomorphism $\psi':E_I \to I \times E'|_{\{1\}}$.  Notice that the composition $\psi'\circ \eta: I\times G \to I \times E'|_{\{1\}}$ defines an identification $\eta':G\to E'|_{\{1\}}$. Consider the composition of bundle maps given by 
$$
\rho:E_0 \to E \stackrel{\psi}{\longrightarrow} I \times E' \to \{1\} \times E' \to I \times E' \stackrel{\psi^{-1}}{\longrightarrow} E
$$
and observe that it fits in the commutative square below.
$$
\xymatrix{
E_{0} \ar[r]^{\rho} \ar[d] & E \ar[d] \\
S^1 \ar[r]^{\varepsilon_1} & I \times S^1
}
$$
By the universal property of Cartesian squares, this induces a unique isomorphism $\varphi:E_0 \to E_1$, which is the desired pointed isomorphism.

\end{proof}

\begin{remark} 
    It is well-known that for a compact Lie group $G$, (unpointed) isomorphism
    classes of principal $G$-bundles over the circle are in bijective
    correspondence with the set of conjugacy classes of $G/G_0$. 

\end{remark}

\begin{lemma}\label{lemma:pi1 topological space}
Let $X$ be a topological space and $p$ a  point in $X$. Then 
\[
\pi_1(\und{X}, \und{p}) \cong \pi_1(X,p)
\]
where the right hand side is the classical fundamental group of a
topological space with chosen basepoint. 
\end{lemma}

\begin{proof}
A pointed map of stacks $(\und{S^1},\und{1}) \to (\und{X},\und{p})$ is
determined uniquely by
  a pointed map  $(S^1, 1) \to (X,p)$ of topological spaces. Since there are no
non-trivial
  $2$-morphisms in  stacks of the form $\und{M}$, classical homotopies are in one-to-one correspondence with stack homotopies, and the claim follows.
\end{proof}

There is a natural morphism $q:\und{X} \to [X/G]$, defined on objects as 
$$
S\to X \quad \mapsto \quad (S \leftarrow S\times G \to X\times G \stackrel{{\mathrm{act}}}{\longrightarrow}  X).
$$
Similarly there is a natural morphism from 
$[X/G]$ to $BG$ which simply forgets the equivariant map to $X$. 
By choosing identifications of the trivial $G$-bundle over a point, these may each 
be considered as pointed  maps.  Finally, recall the pointed map $\iota:G_0 \to X$ given by parametrizing the orbit of the base point $g\mapsto p\cdot g$.
Applying the fundamental group functor to each of these maps and using the isomorphism in the previous Lemma results in the following.

\begin{proposition} 
Let $G$ be a Lie group acting smoothly on a connected manifold $X$.  The following sequence is exact:
\begin{equation}\label{eq:pi1 exact sequence}
\pi_1(G_0,1) \to \pi_1(X,p) \to \pi_1([X/G],x) \to G/G_0\to 1.
\end{equation}
\end{proposition}
\begin{proof}
The verification is straightforward. We prove exactness at $\pi_1(X,p)$ and exactness at $G/G_0$,  leaving the rest for the reader.  Recall first that the trivial element in $\pi_1([X/G],x)$ may be represented by the pair (as in Remark \ref{remark:representingmaps}) 
$
(S^1 \leftarrow S^1 \times G \stackrel{\sigma}{\longrightarrow} X),
$
where $\sigma(z,g) = p\cdot g$.

We  show that the composition of the first two maps in the sequence is trivial.  Let $\alpha\colon S^1 \to G_0$ represent an element of $\pi_1(G_0,1)$, whence its image  via the composition of the first two maps in the sequence is represented by the pair $(S^1\leftarrow S^1 \times G \stackrel{a}{\longrightarrow} X)$ where  $a(z,g) = p\cdot (\alpha(z)g)$. The map $\alpha$ determines a gauge transformation  $\phi:S^1 \times G \to S^1 \times G$ defined by $\phi(z,g)=(z,\alpha(z)g)$. Since $\sigma \circ \phi =a$ we see that $\phi$ determines a 2-isomorphism of  the maps $q\circ \iota \circ \alpha$ and the constant map to $[X/G]$.  By Lemma \ref{lemma:homotopy class independent of 2-morphism} this induces a homotopy, verifying that the composition of the first two maps is trivial. 

Next, suppose that $f:S^1 \to X$ is a pointed map representing a class in $\pi_1(X,p)$ whose image via $q_*$ in $\pi_1([X/G],x)$ is trivial.  We will show that $f$ is homotopic to a composition $\iota \circ \beta$ for some pointed map $\beta:S^1 \to G_0$. 
Suppose that $q\circ f$ is homotopic to the constant map $S^1 \to \star \stackrel{x}{\longrightarrow} [X/G]$.  Let $( I\times S^1 \leftarrow E_H \stackrel{h}{\longrightarrow} X)$ be a pair of maps representing the homotopy $H:I\times S^1 \to [X/G]$.  That $E_H$ represents a homotopy implies that there exists a trivialization $\epsilon_0: S^1 \times G \to E_{H_0}= E_H|_{0\times S^1}$ fitting in the commutative diagram,
$$
 \xymatrix@R=3em{
& & & X \\
{S^1 \times G}\ar@/^1pc/[urrr]^(.4){(z,g) \mapsto f(z) \cdot  g \phantom{X}} \ar[dr] \ar[r]^-{\epsilon_{0}}& E_{H_0} \ar[r] \ar[d] & E_H \ar[d] \ar@/_/[ur]_{h} \\
 & {S^1} \ar[r]^-{\varepsilon_0} & {I \times S^1 } \\
 }
$$
where $\varepsilon_0(z) = (0,z)$
Note that by Theorem 9.8 in \cite[Ch. 4]{Husemoller:1994}, the bundle $E_H \to I \times S^1$ is trivializable.  Moreover, we may choose a section $s$ so that the composition
$$
S^1 \stackrel{s\varepsilon_0}{\longrightarrow} E_{H_0} \stackrel{\epsilon_0^{-1}}{\longrightarrow} S^1 \times G
$$
is simply inclusion $z\mapsto (z,1)$. 

Since $E_H$ represents a homotopy, there exists a trivialization $\epsilon_1:E_{H_1}=E_H|_{1\times S^1} \to S^1 \times G$.  Note that $\epsilon_1\circ s\circ \varepsilon_1 (z) = (z,\beta(z))$ for some loop $\beta:S^1 \to G$, and that $\beta(1) \in \Stab(p)$ because the homotopy $H$ is a homotopy of pointed maps.  By replacing the 2-isomorphism 
$\epsilon_1:E_{H_1} \to S^1 \times G$ (that is part of the data of the homotopy $H$)
with the composition of  2-isomorphisms below if necessary,
$$
\xymatrix@R=3em{
X & & & \\
&{E_{H_1}}\ar[dr] \ar@/^/[ul] \ar[r]^{\epsilon_1} & {S^1 \times G}\ar@/_/[ull]_{\sigma} \ar[d] \ar[r]^{L_{\beta(1)^{-1}}} & {S^1\times G}\ar@/_2pc/[ulll]_{\sigma} \ar[dl] \\
& & S^1 & 
}
$$
(where  $L_{u}$ denotes the gauge transformation $(z,g) \mapsto (z,ug)$) we may assume that $\beta(1)=1$.

The map $\tau=h\circ s:I\times S^1 \to X$ is the desired homotopy.  It is readily verified that $\tau(0,z)=f(z)$, and that  $\tau(1,z) = p\cdot \beta(z) = (\iota \circ \beta) (z)$, as required.  

Finally, suppose $E(g) \to S^1$ represents an
element of $\pi_1(BG,*) \cong G/G_0$. It suffices to construct an equivariant
map $E(g) \to X$. Choose a point $z \in X$ and a path $\gamma:I \to X$ with
$\gamma(0)=z$ and $\gamma(1)=z\cdot g^{-1}$. The map $I \times G \to X$
given by $(t,h) \mapsto \gamma(t)\cdot h$ descends to $E(g)$ and is
$G$-equivariant.
\end{proof}

\begin{corollary} \label{cor:orb-pi-one when X simply connected}
Let $G$ be a Lie group acting smoothly on a  simply connected
manifold $X$. The fundamental group $\pi_1([X/G],x) \cong G/G_0$. In particular, if in addition $G$ is
connected, then $\pi_1([X/G],x) $ is trivial.
\end{corollary}


\begin{proposition}\label{proposition:quotient homomorphism}
  Under the identifications $I_x \cong \Stab(p)$ and $\pi_1(BG,*) \cong
G/G_0$ given above, the composition $\Stab(p) \cong I_x
  \stackrel{\omega_x}{\longrightarrow} \pi_1([X/G],x) \to G/G_0$ is the
  natural  homomorphism $ \phi_p:\Stab(p) \hookrightarrow G
  \to G/G_0$.
\end{proposition}
\begin{proof}
For $g \in \Stab(p)$, the image of $\omega_x(g)$ in $\pi_1(BG,*)$ is represented
by the bundle $E(g)$ (in the notation above), which by the proof of Lemma
\ref{lemma:fundamental group of BG} is represented by the coset $g G_0$.
\end{proof}

\section{Global quotients and universal coverings} \label{section:quotients and covers}

A main purpose of this paper is to determine conditions under which a Deligne-Mumford stack is a global quotient in the sense of the following definition.

\begin{definition} \label{def:discrete/global quotient}
A  Deligne-Mumford stack $\mathcal{X}$ is a \emph{discrete} (resp. \emph{global}) \emph{quotient} if $\mathcal{X}$ is equivalent to a quotient stack $[Y/\Gamma]$, where $\Gamma$ is a discrete (resp. finite) group acting   on a smooth manifold $Y$.
\end{definition} 

Recall that a morphism of representable stacks $\mathcal{X} \to \mathcal{Y}$ is an \emph{equivalence} if it is an equivalence of categories.  An equivalence may also be represented in terms of the corresponding representing groupoids as a principal bi-bundle (e.g. see  Definition 3.25 and Remark 3.33 in \cite{Lerman:2010}).  

We shall deal only with the special case of quotient stacks $\mathcal{X}=[X/G]$ arising from a smooth proper action of a   Lie group $G$ on a connected manifold $X$.  Note that an equivalence of such a pair of quotient stacks $[X/G] \to [Y/H]$ may then be represented as a bi-bundle of the action groupoids $X\times G \rightrightarrows X$ and $Y\times H \rightrightarrows Y$, which in this case amounts to a $G\times H$-space $P$ that is simultaneously a principal $G$-bundle $P\to Y$ (with $H$-equivariant projection) and a principal $H$-bundle $P\to X$ (with $G$-equivariant projection).

\subsection{Global quotients and the fundamental group}

The following Lemma provides a natural setting to discuss a class of examples of quotient stacks that are equivalent to global quotients.

\begin{lemma}  \label{lemma:quotient in stages}
Let $ 1 \to H \to G \to \Gamma \to 1 $
be an exact sequence of  topological (respectively, Lie) groups. 
Suppose that  $G$ acts on a topological space (resp.  smooth manifold) $X$ and that the restriction of this action to $H$ is  free (resp. free and proper).  
Then $[X/G]$ and $[(X/H)/\Gamma]$ are
equivalent as stacks over $\Top$ (resp. $\mathsf{Diff}$). 
\end{lemma} 
\begin{proof}
We shall work over $\Diff$, noting that the proof is the same over $\Top$.
Since the $H$-action is free and proper, the orbit space $X/H$ is indeed a smooth manifold.
Define $F:[X/G] \to [(X/H)/\Gamma]$ to be the functor defined by the
assignment
$$
( B\leftarrow  E\to X) \mapsto ( B\leftarrow E/H \to X/H)
$$ 
on objects, and 
$$
\xymatrix@C=1.5em@R=1.5em{
 & & {X} \\
 {E} \ar@/^/[urr] \ar[r] \ar[d] & {E'} \ar[d] \ar@/_/[ur] & \\
 {B} \ar[r] & {B'} & \\
}\quad
\xymatrix@C=1.5em@R=1.5em{
  &  \\
 \mapsto  & \\
  & \\
}
\xymatrix@C=1.5em@R=1.5em{
 & & {X/{H}} \\
 {E/{H}} \ar@/^/[urr] \ar[r] \ar[d] & {E'/{H}} \ar[d] \ar@/_/[ur] & \\
 {B} \ar[r] & {B'} & \\
}
$$
on arrows. By construction, $F$ commutes with the projections to the
base category $\mathsf{Diff}$.  We wish to show that $F$ is an
equivalence of categories. To see this, we define a functor $K:
[(X/H)/\Gamma] \to [X/G]$ as follows. 
Suppose given a pair $(B\leftarrow P\to X/H)$ consisting
of a $\Gamma$-bundle $P\to B$ and a $\Gamma$-equivariant map $P\to X/H$.  Let
$E:=P\times_{X/H} X$ be the fiber product, and define a $G$-action on $E$ by
setting $(p,z)\cdot g = (p\cdot \rho(g), z\cdot g)$, where $\rho:G\to
\Gamma$ denotes
the given map in the exact sequence.  
We claim that the composition 
$E\stackrel{\operatorname{pr}_1}{\longrightarrow} P \to B$ 
is a principal $G$-bundle. Indeed suppose $(p, z)\cdot g = (p \rho(g),
zg) = (p,z)$. Then since $P$ is a principal $\Gamma$-bundle, $\rho(g)$
is the identity element in $\Gamma$ and hence $g \in H$. On the other
hand, by assumption $H$ acts freely on $X$ so $zg=z$ implies
$g=\textup{id}$ in $G$. Hence $G$ acts freely on $E$. Next suppose
$(p,z) \in E, (p', z') \in E$ map to the same point $b$ in $B$. Then
since $P$ is a principal $\Gamma$-bundle over $B$, there exists $g \in
G$ such that $p' = p \rho(g)$. By definition of the fiber product $E =
P \times_{X/H} X$, the equivariance of the map $P \to X/H$, and
normality of $H$, we conclude there exists $h \in H$ such that $z' = z
hg$. Since $\rho(hg)=\rho(g) \in \Gamma$ we conclude $(p,z)hg = (p',
z')$ and that $G$ acts transitively on fibers of $E \to B$. Hence $E
\to B$ is a principal $G$-bundle, as desired. The projection map $E \to
X$ is $G$-equivariant by construction so $(B\leftarrow E \to X)$ is an
object in $[X/G]$. 
Given an arrow 
$$
\xymatrix@C=1.5em@R=1.5em{
 & & {X/{H}} \\
 {E/{H}} \ar@/^/[urr] \ar[r] \ar[d] & {E'/{H}} \ar[d] \ar@/_/[ur] & \\
 {B} \ar[r] & {B'} & \\
}
$$
in $[(X/H)/\Gamma]$, it induces a unique arrow
$$
\xymatrix@C=1.5em@R=1.5em{
 & & {X} \\
 {E/{H}\times_{X/H} X} \ar@/^/[urr] \ar[r] \ar[d] & {E'/{H}\times_{X/H}
X} \ar[d] \ar@/_/[ur] & \\
 {B} \ar[r] & {B'} & \\
}
$$
which defines the functor $K$ on morphisms. Again by construction $K$
commutes with projection to the base category. 

Finally, we sketch the constructions of the natural transformations
between $F \circ K$ (resp. $K \circ F$) and the identity functor on
$[(X/H)/\Gamma]$ (resp. $[X/G]$). For $K \circ F$, observe 
that for any object $(B\leftarrow E\to X)$ in $[X/G]$, there is a unique
isomorphism
$$
\xymatrix@C=1.5em@R=1.5em{
 & & & {X} \\
 {E} \ar@/^/[urrr] \ar[rr] \ar[dr] & & {{E/H}\times_{X/H} X}
\ar[dl] \ar@/_/[ur] & \\
 & B &  & \\
}
$$
in $[X/G]$.  For $F\circ K$, observe that $(P \times_{X/H} X)/H$ is
isomorphic to $P$ via the map $[(p,z)] \mapsto p$. From here it is 
straightforward to check that these yields the desired natural
transformations. 
\end{proof}

Observe that the requirement in Lemma \ref{lemma:quotient in stages} that the $G$-action restricted to $H$ be a free $H$-action is necessary over both $\Diff$ and $\Top$. (Compare with Proposition \ref{prop:stackquotientstages}.) For example, consider the exact sequence
$$
1 \to \Z_2 \to S^1 \stackrel{2}{\longrightarrow} S^1 \to 1
$$
where $2:S^1 \to S^1$ denotes the squaring map.  Let $t\in G=S^1$ act on $X=S^3 \subset \C^2$ according to $t\cdot(z,w)=(t^2z, t^2w)$.  The quotient stack $[X/G]$ is the weighted projective space $\mathbb{P}(2,2)$.  The restriction of the $G$-action to $H=\Z_2=\{\pm 1\}$ is trivial, and the resulting residual action of $\Gamma =S^1$ on $X/H=X=S^3$ is the standard action of $S^1$ on $S^3$ giving the quotient $[(X/H)/\Gamma]=\mathbb{P}(1,1)$, the complex projective plane.  As  a stack, $\P(2,2)$ has a non-trivial  inertia group isomorphic to $\Z_2$ at each point and is thus not equivalent to the smooth manifold $\P(1,1)$.

If $G$ is compact, applying the above Lemma to the case $H=G_0$, the
connected component of the identity element, provides a natural class
of examples of stacks equivalent to global quotients. (Since $G$ is
compact, the quotient $G/G_0$ is automatically a finite group.)  

\begin{corollary}\label{cor:quotient in stages}
Suppose a Lie group $G$ acts smoothly and properly on a connected 
smooth manifold $X$. Let $\Gamma$ denote the (discrete)
group $G/G_0$ of $G$. If the restriction of the $G$-action to $G_0$ is free, then
$[X/G]$ and $[(X/G_0)/\Gamma]$ are equivalent as 
stacks over $\mathsf{Diff}$ and hence $[X/G]$ is a discrete quotient.  If in addition $G$ is compact, then $[X/G]$ is a global quotient. 
\end{corollary}

If in addition $X$ is simply connected, then we shall see in Theorem \ref{thm:TFAE} that the above examples characterize global quotients among quotient stacks.  
Proposition \ref{proposition:separated stabilizers} below illustrates how the freeness of the  $G_0$-action  on $X$ relates to the fundamental group $\pi_1([X/G],x)$.

\begin{proposition}\label{proposition:separated stabilizers}
Suppose a  Lie group $G$ acts smoothly on a connected 
smooth manifold $X$. If the restriction of the $G$-action to $G_0\subset G$ is free, then the homomorphism $\omega_x$ is injective for all points $x$ in $[X/G]$.  Moreover, if $X$ is simply connected, the converse holds as well.
\end{proposition}

\begin{proof} 
  From Proposition~\ref{proposition:quotient homomorphism} we know that the
  composition of the homomorphism $\omega_x$ with the second arrow
  in~\eqref{eq:pi1 exact sequence} is  the natural homomorphism
  $\phi_p:\Stab(p)  \to G/G_0$ obtained as the composition of the
  natural inclusion $\Stab(p) \into G$ with the canonical quotient map $G \to G/G_0$.  If the restriction of the $G$-action to $G_0$ is free, then $\phi_p$ is
  injective for all $p$ and hence $\omega_x$ is injective for all $x$ in $[X/G]$. This proves the first claim. 

  For the second claim, if $X$ is simply connected then 
  by Corollary~\ref{cor:orb-pi-one when X simply connected}
  $\pi_1([X/G],x) \cong \pi_1(BG,*) \cong G/G_0$ and $\ker \omega_x = \Stab(p) \cap G_0$.  Therefore, if $\omega_x$ is injective for all $x$ in $[X/G]$ then $G_0$ acts freely on $X$.
  \end{proof}

\subsection{On covers of quotient stacks}
Lemma \ref{lemma:quotient in stages} may be generalized to the context
of group actions on stacks, which then fits nicely with covering
theory.  In preparation for the statement of Proposition~\ref{prop:stackquotientstages}, we begin with a summary of some ideas found in the work of
Lerman and Malkin \cite{LermanMalkin2009}, which the reader should
consult for details. 

For a Lie group $\Lambda$, recall that a $\Lambda$-action on a stack $\calX$ can
be encoded using a \emph{$\Lambda$-presentation}, a groupoid presentation
$\mathcal{G}~=\mathcal{G}_1 \rightrightarrows\mathcal{G}_0$ of $\calX$
equipped with smooth and free $\Lambda$-actions on both the manifold of
arrows $\mathcal{G}_1$ and the manifold of objects $\mathcal{G}_0$
that is compatible with the structure maps of the groupoid
$\mathcal{G}$.

Towards generalizing Lemma \ref{lemma:quotient in stages}, suppose that 
$$
1 \to H \into G \to \Gamma \to 1 
$$
is an exact sequence of Lie groups and that $G$ acts properly and smoothly on a manifold $X$.  The exact sequence above naturally defines a $G$-action on $X\times \Gamma$ and the translation groupoid 
\begin{equation}
G\times (X\times \Gamma) \rightrightarrows X\times \Gamma \label{eqn:presentation}
\end{equation}
is then a groupoid presentation for the quotient stack $[(X\times \Gamma) /G]$.  As in \cite[Section 4.1]{LermanMalkin2009}, since translation by $\Gamma$ commutes with the above $G$-actions, we see that (\ref{eqn:presentation}) is a $\Gamma$-presentation.  By \cite[Proposition 4.2]{LermanMalkin2009} (\ref{eqn:presentation}) is also a $\Gamma$-presentation for the quotient stack $[X/H]$.  This gives a $\Gamma$-action on $[X/H]$, which (see \cite[Section 3.3]{LermanMalkin2009}) shows that  the translation groupoid $G\times X \rightrightarrows X$
is a groupoid presentation of the stack quotient  $[X/H]/\Gamma$.
This verifies the following generalization of Lemma \ref{lemma:quotient in stages}.

\begin{proposition} \label{prop:stackquotientstages}
Let $ 1 \to H \to G \to \Gamma \to 1 $
be an exact sequence of Lie groups. 
Suppose that  $G$ acts properly and smoothly on a  smooth manifold $X$.  Then the quotient stack $[X/H]$ inherits a $\Gamma$-action; moreover,   $[X/G]$ and $[X/H]/\Gamma$ are
equivalent as stacks over  $\mathsf{Diff}$.
\end{proposition}

As in previous discussions, we wish to interpret the above Proposition in the case $H=G_0$.  This interpretation may be placed in the context of covering theory for stacks \cite{Noohi:2005}.  In particular, we shall see in Proposition \ref{prop:covering} that the natural map $p:[X/G_0] \to [X/G]$ is a covering projection.  (More generally, there is a natural map $[X/H] \to [X/G]$, given by the associated bundle construction, which is representable by Lemma \ref{lemma:associatedbundle} below.)   In other words, we may view $p$ as a quotient map, for $[X/G] \cong [X/G_0]/\Gamma$ where $\Gamma$ is the discrete group $G/G_0$.

Parallel to classical covering space theory, one may define universal covering projections.  For simplicity, we shall define a \emph{universal covering projection} $(\tilde{\calX},\tilde{x})\to(\calX,x)$,  as a covering projection with $\pi_1(\tilde{\calX},\tilde{x})=\{1\}$
  (Cf. \cite[Corollary 18.20]{Noohi:2005}).

\begin{lemma} \label{lemma:associatedbundle}
Let $H$ be a closed subgroup of a Lie group $G$ that acts  smoothly on a manifold $X$. The natural map $[X/H] \to [X/G]$ given by the associated bundle construction is representable.
\end{lemma}
\begin{proof}
Let $\varphi:\und{W} \to [X/G]$ and let $(W\leftarrow E_\varphi \to
X)$ denote $\varphi(\id_W)$.  Recall that  the fiber product $\mathcal{Z}=[X/H]
\times_{[X/G]} \und{W}$ has objects given by triples $( B\leftarrow E
\to X, f:B\to W, \alpha) $ where $\alpha\in [X/G]_{B} $ is an isomorphism 
 of $G$-bundles $\alpha:E\times_H G \to f^*E_\varphi$ (compatible with the maps to $X$).  An arrow in $\mathcal{Z}$ between two such objects is an arrow  $(E\to E', B\to B')$ in $[X/H]$ such that $B\to B'$ is compatible with the maps to $W$ and that the resulting (vertical) induced maps in the diagram below commute.
$$
\xymatrix@C=1.5em@R=1.5em{
 E\times_H G \ar[r]^{\alpha} \ar[d] & f^*E_\varphi \ar[d] \\
 E'\times_H G \ar[r]^{\alpha'} & (f')^*E_\varphi \\
 }
$$

Since $E_\varphi \to W$ is a principal $G$-bundle, the $G$-action on $E_\varphi$ is free and proper; therefore, the restriction of this action to $H\subset G$ is also free and proper, whence $E_\varphi/H$ is a manifold in $\Diff$.
Define  $F:\mathcal{Z}\to \und{E_\varphi/H}$ as follows.  
Using the section $B\to (E\times_H G)/H$, which sends $b\in B$ to the $H$-orbit of $[(e,1)]$ where $e\in E$ is any element in the fiber over $b$, and composing with the induced composition 
$$
(E\times_H G)/H \stackrel{\alpha/H}{\longrightarrow} f^* E_\varphi/H \to E_\varphi/H
$$
we obtain a map $B\to E_\varphi /H$.  Hence, on objects, 
we define $F( B\leftarrow E \to X, f:B\to W, \alpha) = (B\to E_\varphi/H)$.  (The effect of $F$ on arrows is the natural one.) 

To show that $F$ is an equivalence, we next define a morphism $K:\und{E_\varphi/H} \to \mathcal{Z}$ as follows. Given a $B\to E_\varphi/H$, let $P$ denote the pullback $H$-bundle 
$$
\xymatrix@C=1.5em@R=1.5em{
 P \ar[r] \ar[d] & E_\varphi \ar[d] \\
 B \ar[r] & E_\varphi/H \\
 }
$$
let $f$ denote the composition $f:B\to E_\varphi/H \to W$.  Since the map $P\times_H G \to E_\varphi$ given by $[(b,e,g)] \mapsto e\cdot g$ covers $f$, there is a unique isomorphism $P\times_H G \to f^*E_\varphi$ which we denote by $\alpha$.  Hence, on objects,  we define $K(B\to E_\varphi/H)$ to be the triple
$(B\leftarrow P \to E_\varphi \to X, f:B\to W, \alpha)$.  (The effect of $K$ on an arrow $B\to B'$ in $\und{E_\varphi/H}$ is hence determined.)

It is straightforward to verify that $F\circ K$ is the identity.  To realize the natural transformation between $K\circ F$ and the identity, we simply note that pulling back an $H$-bundle via a composition yields a canonical bundle isomorphism, so that the first factor in the triple for  $K\circ F( B\leftarrow E \to X, f:B\to W, \alpha)$ is thus canonically isomorphic to $ B\leftarrow E \to X$, which results in the desired natural transformation. 
\end{proof}

\begin{proposition}\label{prop:covering}
Suppose a Lie group $G$ acts on a smooth manifold $X$ and let $G_0$ denote the identity component of $G$.  The natural map $p:[X/G_0] \to [X/G]$ is a covering projection.  Moreover, if $X$ is simply connected, $p$ is the universal covering projection.
\end{proposition}

\begin{proof}
  From Lemma \ref{lemma:associatedbundle}, the natural map $p$ is
  representable, and the proof of the Lemma shows that given $\und{W}
  \to [X/G]$, that the induced $G/G_0$-bundle $E_\varphi/G_0 \to W$ (a
  covering projection) is a representative for $p$. If $X$ is simply
  connected, by Corollary \ref{cor:orb-pi-one when X simply
    connected}, $p$ is the universal covering projection.
\end{proof}


\begin{remark} 
The above proposition identifies the universal cover of the quotient
stack $[X/G]$ in the setting when $X$ is simply connected. 
In Proposition \ref{prop:nonsimplyconnectedcase} below, we identify
the universal cover of $[X/G]$ when $X$ is \emph{not} simply
  connected. 
\end{remark}

\subsection{Characterizations of global quotients among quotient stacks of simply connected manifolds $X$}

We now state the main result of this section, which in a sense also summarizes the previous subsections.
 Theorem \ref{thm:TFAE}  below characterizes 
discrete (resp. global) quotients among quotient stacks of simply connected manifolds.  (The reader may wish to compare with  \cite[Theorem 18.24]{Noohi:2005}, which discusses a more general setting.)

\begin{theorem} \label{thm:TFAE}
Let $X$ be a simply connected manifold, equipped with a smooth proper action of a Lie group $G$.  Let $G_0\subset G$ denote the connected component of the identity, and $I_x$ the inertia group of $x \in [X/G]$. The following statements are equivalent.
\begin{enumerate}
\item $[X/G]$ is equivalent to a discrete quotient.
\item $G_0$ acts freely on $X$.
\item $\omega_x:I_x \to \pi_1([X/G],x)$ is injective for all $x$ in $[X/G]$.
\item The universal cover of $[X/G]$ is equivalent to a smooth manifold.
\end{enumerate}
\end{theorem}
\begin{proof}
Most implications follow directly from work in previous sections:
 (2) $\Rightarrow$ (1) is Corollary~\ref{cor:quotient in stages};
(3) $\Leftrightarrow$ (2) follows from Proposition~\ref{proposition:separated stabilizers}.  
By Proposition \ref{prop:covering}, the universal cover of $[X/G]$ is $[X/G_0]$, which verifies  (4)~$\Rightarrow$~(2).  Conversely, if $G_0$ acts freely on $X$, the principal $G_0$-bundle $X \to X/G_0$ may be viewed as a bi-bundle equivalence $[X/G_0] \cong \underline{X/G_0}$, and hence (2) $\Rightarrow$ (4).

It remains to show
(1) $\Rightarrow$ (2).  To that end, suppose given a principal bi-bundle representing an equivalence $[X/G] \cong [Z/\Lambda]$ where $\Lambda$ is discrete.  Recall that this yields a $G\times \Lambda$-space $P$ that is simultaneously a $G$-bundle $P\to Z$ (with $\Lambda$-equivariant projection) and a $\Lambda$-bundle $P\to X$ (with $G$-equivariant projection).   Since $X$ is simply connected, we have $P\cong X\times \Lambda$. And since  the $G$ and $\Lambda$-actions commute, the $G$-action on $X\times \Lambda$ may be written
$$
g\cdot (x,\lambda) = (g\cdot x, \phi(g,x)\lambda)
$$
where in the first factor $\cdot$ signifies the original $G$-action on $X$. Since $\Lambda$ is discrete, $\phi$ only depends on the component of $g$, yielding a homomorphism $\varphi:G/G_0 \to \Lambda$ with $\varphi(gG_0) = \phi(g,x)$ for any $x$. Finally, if $g\in G_0$ stabilizes $x$ in $X$, then $g\cdot (x,\lambda) = (x,\lambda)$ and hence $g$ is the identity element, as required.
\end{proof}

The condition that $X$ be simply connected in the above theorem is necessary, as illustrated by the following example. 

\begin{example}\label{eg:wt 2 action on circle}
 Let $G=S^1$ act on $X=S^1$ with weight 2.  This action has a global stabilizer $\Z_2=\{\pm 1\} \subset G_0=G$.  Nevertheless, we may readily verify that $[X/G]\cong B \Z_2 = [\star / \Z_2]$.  Explicitly, consider the following functors $F$ and $K$. On objects, let $F:B\Z_2 \to [X/G]$ be defined by taking associated bundle 
 $$F(E\to B) = (B \leftarrow E\times_{\Z_2} S^1 \stackrel{f}{\longrightarrow} X)$$ with $f([(e,z)])=z^2$, and let $K:[X/G] \to B\Z_2$ be given   by $K(B\leftarrow P\stackrel{f}{\longrightarrow} S^1) = (f^{-1}(1) \to B)$. (The effects of $F$ and $K$ on arrows are the natural ones.)   Alternatively, see Example \ref{eg:wt 2 action on circle redux} and Proposition \ref{prop:nonsimplyconnectedcase} below.
\end{example}

\subsection{The universal cover of a quotient stack of non-simply connected manifold $X$}

We next work towards a statement in the spirit of the equivalence (1)~$\Leftrightarrow$~(2) of  Theorem \ref{thm:TFAE} for the case of connected manifolds $X$ that are not necessarily simply
connected.  Let $X$ and $G$ be as in Proposition \ref{prop:covering},
and let $\tilde{X}$ denote the universal cover of a smooth manifold
$X$ and consider the induced action of $\tilde{G}_0$ on $\tilde{X}$.
Let $\Lambda$ denote the image of $\pi_1(G_0,1) \to \pi_1(X,p)$
(\emph{cf}. (\ref{eq:pi1 exact sequence})).  Since $\Lambda$ is a
subgroup of deck transformations, $\tilde{X}/\Lambda$ is a smooth
manifold.  Moreover, $G_0=\tilde{G}_0/\pi_1(G_0,1)$ acts on
$\tilde{X}/\Lambda$, and the covering projection $\tilde{X}/\Lambda
\to X$ is $G_0$ equivariant.

The following technical lemma helps to identify the natural map  $[(\tilde{X}/\Lambda)/G_0] \to [X/G_0]$ as a covering projection in the following Proposition.

\begin{lemma} 
  Suppose a Lie group $G$ acts smoothly on connected manifolds
  $Y$ and $X$. Assume the $G$-action on $Y$ is proper and that $f:Y\to
  X$ is a $G$-equivariant submersion. 
  Then
  the canonical map $[Y/G] \to [X/G]$ is representable.
\end{lemma}

\begin{proof}
  Let $\varphi:\und{W} \to [X/G]$ be given and let $( W\leftarrow
  E_\varphi \to X)$ denote $\varphi(\id_W)$.  The fiber product 
  $Y\times_{X} E_\varphi$ is a smooth $G$-invariant manifold since $Y \to X$ is a
  $G$-equivariant submersion. Moreover, the diagonal $G$-action on $Y
  \times_X E_\varphi$ is proper since the $G$-actions on $Y$ and
  $E_\varphi$ are both proper. 
  %
  Further, since the
  canonical map $Y\times_{X} E_\varphi \to E_\varphi$ is
  $G$-equivariant and $G$ acts freely on $E_\varphi$, then $G$ also acts
  freely on $Y\times_{X} E_\varphi$.  By the slice theorem for proper
  $G$-actions, we conclude that the quotient $U=(Y\times_{X} E_\varphi
  )/G$ is a smooth manifold.

We claim that $Z=[Y/G] \times_{[X/G]}
  \und{W}\cong \underline{U}$.
  Recall that the objects of $Z$ are triples $(B\leftarrow E \to Y,
  f:B\to W, \alpha)$ where $\alpha:E\to f^*E_\varphi$ is a bundle
  isomorphism compatible with the equivariant maps to $X$. An object
  in $Z$ therefore determines a canonical map $E\to Y\times_X
  E_\varphi$ that is $G$-equivariant.  This map descends to a map
  $B\to U$, which defines a functor $F:Z\to \und{U}$.

  Towards showing that $F$ is an equivalence, we next define
  $K:\und{U} \to Z$.  Given a map $B\to U$, by pulling back and
  composing with the natural projection, we obtain $B\leftarrow P \to
  Y\times_X E_\varphi \to Y$, and we may set $K(B\to U)$ to be the
  triple $(B\leftarrow P \to Y, B\to U \to W, \alpha)$ where $\alpha$
  is the canonical isomorphism given by pulling back along a
  composition.

  That $F\circ K$ is the identity is easily verified.  Similarly, the
  composition $K\circ F$ is canonically isomorphic to the identity
  functor.
\end{proof}

\begin{proposition} \label{prop:nonsimplyconnectedcase} Suppose a
  compact Lie group $G$ acts smoothly on a connected manifold $X$.
  The quotient stack $[X/G]$ is equivalent to a quotient of a discrete
  group action over $\mathsf{Diff}$ if and only if the restriction of
  the induced $G_0$-action on $\tilde{X}/\Lambda$ is free, where
  $\Lambda$ denotes the image of $\pi_1(G_0,1) \to \pi_1(X,p)$
  . Additionally, the composition of natural maps
  $[(\tilde{X}/\Lambda)/G_0] \to [X/G_0] \to [X/G]$ is a universal
  covering projection.
\end{proposition}
\begin{proof}
By  \cite[Theorem 18.24]{Noohi:2005} $[X/G]$ is equivalent (over $\Diff$) to a quotient by a discrete group action if and only if its universal cover is equivalent to a manifold.  By Proposition \ref{prop:covering}, it suffices to determine conditions under which the universal cover of $[X/G_0]$ is equivalent to a manifold, which is done next.

The representable map $p:[(\tilde{X}/\Lambda)/G_0] \to [X/G_0]$ is a covering projection.  Indeed, by the proof of the previous Lemma, given $\varphi:\und{W} \to [X/G_0]$, the natural projection $(\tilde{X}/\Lambda) \times_{X} E_\varphi \to E_\varphi$ is a $G_0$-equivariant covering projection, which induces $((\tilde{X}/\Lambda) \times_{X} E_\varphi)/G_0\to E_\varphi/G_0=W$, a covering projection that  represents $p$.

Applying the exact sequence (\ref{eq:pi1 exact sequence}) to $[(\tilde{X}/\Lambda)/G_0]$, and noting that the first map in this exact sequence is a surjection, we see that $[(\tilde{X}/\Lambda)/G_0]$ is the universal cover of $[X/G_0]$.  Finally,
$[(\tilde{X}/\Lambda)/G_0]$ is equivalent to a manifold if and only if (the compact group) $G_0$ acts freely on $\tilde{X}/\Lambda$.
\end{proof}

%

\begin{example}[Example \ref{eg:wt 2 action on circle}, revisited] \label{eg:wt 2 action on circle redux} If $G=S^1=\R/\Z$ acts on $X=S^1=\R/Z$ with weight 2, then $\Lambda = 2\Z \subset \Z = \pi_1(X,1)$ and the induced $G$-action on $\tilde{X}/\Lambda =\R/ \Lambda$ may be written $e^{2\pi i \theta} \cdot e^{\pi i t} = e^{\pi i(t+2\theta)}$, which is free (and transitive).  Therefore, as in the proof of the previous proposition, the universal cover of $[X/G]$ is $[(\R/\Lambda)/G]=\star$ and $[X/G] \cong [\star/\Z_2]$.
\end{example}

\section{Toric DM stacks} \label{sec:toricDMstacks}

We now apply the ideas of the previous section to toric Deligne-Mumford stacks arising from the combinatorial data of  \emph{stacky fans} \cite{BCS05} and \emph{stacky polytopes} \cite{Sakai2010}.  As we shall review below, these stacks arise as quotients $[X/G]$ of a simply connected spaces $X$; therefore, we may apply Theorem \ref{thm:TFAE}.


\subsection{Stacky fans and polytopes---brief review}
Mainly to establish notation, we  briefly recall some basic definitions of the combinatorial data appearing in the above discussion. In the following we use $( - )^\dual$ to denote the functor $\mathrm{Hom}_{\Z}(-,\Z)$ or $\mathrm{Hom}_\R(-,\R)$; it should be clear from context which one is meant. We use angled brackets $\langle -,- \rangle$ to indicate a natural pairing defined by duality. Also, $-\otimes -$ signifies $-\otimes_\Z-$.

Let $\{ e_1, \ldots, e_n\}$ be the standard basis vectors in $\Z^n \subset \R^n$.

\begin{definition} \cite{BCS05} \label{definition:stacky fan}
A  \emph{stacky fan} is a triple $(N, \Sigma, \beta)$ consisting of a rank~$d$
finitely generated abelian group~$N$, a rational simplicial fan $\Sigma$ in
$N\otimes \R$ with rays $\rho_1, \ldots, \rho_n$ and a homomorphism $\beta:\Z^n
\to N$ satisfying:
\begin{enumerate}
\item \label{item:rays span} the rays $\rho_1, \ldots, \rho_n$ span $N\otimes \R$, and
\item \label{item:beta surjects onto rays} for $1\leq j \leq n$, $\beta(e_j)\otimes 1$ is on the ray $\rho_j$.
\end{enumerate}
\end{definition}

Given a polytope $\Delta \subseteq \R^d$, recall that the
corresponding fan $\Sigma = \Sigma(\Delta)$ is obtained by setting the one dimensional cones $\Sigma^{(1)}$ to be the positive rays spanned by the inward-pointing normals to the facets of
$\Delta$; a subset $\sigma$ of these rays is a cone in $\Sigma$
precisely when the corresponding facets intersect nontrivially in
$\Delta$. Observe that under this correspondence, facets intersecting in a vertex of $\Delta$ yield maximal cones (with respect to inclusion) in $\Sigma(\Delta)$.  


\begin{definition}\label{definition:stacky polytope} \cite{Sakai2010}
 A \emph{stacky polytope} is a triple $(N, \Delta, \beta)$ consisting of a
rank~$d$ finitely generated abelian group~$N$, a simple polytope
$\Delta$ in $(N\otimes  \R)^\dual$ with $n$ facets $F_1, \ldots, F_n$ and a
homomorphism $\beta:\Z^n \to N$ satisfying:
\begin{enumerate}
 \item \label{item:finite cokernel} the cokernel of $\beta$ is finite, and
 \item \label{item:beta hits normals} for $1\leq j \leq n$,  $\beta(e_j) \otimes  1$ in
$N\otimes  \R$ is an inward pointing normal to the facet $F_j$.
\end{enumerate}
\end{definition}

Condition \ref{item:beta hits normals} above implies that the polytope $\Delta$ in Definition \ref{definition:stacky polytope} is a rational polytope.  Also, from   the preceding
discussion it follows immediately that the data of a stacky polytope
$(N,\Delta, \beta)$ specifies the data of a stacky fan by the
correspondence $(N,\Delta, \beta) \mapsto (N, \Sigma(\Delta),\beta)$.
Indeed,  $\Delta$ is simple if and only if $\Sigma(\Delta)$ is simplicial.  Moreover, the fan $\Sigma(\Delta)$ is rational by condition  \ref{definition:stacky polytope} (\ref{item:beta hits normals}).  Finally, $(N,\Delta,\beta)$ satisfies 
conditions (\ref{item:finite cokernel}) and
(\ref{item:beta hits normals}) of Definition \ref{definition:stacky polytope} if and only if 
$(N,\Sigma(\Delta),\beta)$ satisfies  conditions (\ref{item:rays span}) and (\ref{item:beta surjects onto rays})
of Definition \ref{definition:stacky fan}.

The extra information encoded in a stacky polytope $(N,\Delta,\beta)$ (compared with the stacky fan $(N,\Sigma(\Delta),\beta)$) results in a symplectic structure on the associated toric DM stack.  
Given a presentation of a rational polytope $\Delta$ as the intersection of half-spaces
\begin{align} \label{equation:polytope is rational}
 \Delta = \bigcap_{i=1}^n \left\{ x \in (N \otimes  \R)^\dual \,|\,  \langle
x, \beta(e_{i})\otimes 1 \rangle  \geq -c_i \right\}
\end{align}
for some $c_i \in \R$ and where each $\beta(e_{i})\otimes 1\in
N \otimes \R$  is the inward pointing normal to the facet $F_i$, the
fan $\Sigma(\Delta)$ only retains the data of the positive ray spanned
by the normals, and not the parameters $c_i$, which encode the symplectic structure on the resulting DM stack (see \cite{Sakai2010} for details).

%

Stacky polytopes can be thought of as generalizations of Lerman and Tolman's labelled polytopes.
In its original form \cite{LT97},  a labelled polytope is a pair $(\Delta,
\{m_i\}_{i=1}^n)$ consisting of a  convex simple polytope  $\Delta$ in $(N
\otimes \R)^\dual$, where $N$ is a lattice,  with $n$ facets $F_1, \ldots, F_n$ whose relative interiors are
labelled with positive integers $m_1, \ldots, m_n$. If we denote the primitive
inward pointing normals $\nu_1\otimes 1, \ldots, \nu_n\otimes 1$, then defining
$\beta:\Z^d \to N$ by the formula   $\beta(e_i) = m_i \nu_i$ realizes
$(N,\Delta,\beta)$ as a stacky polytope.
Thus  labelled polytopes are precisely the subset of the
stacky polytopes for which the $\Z$-module $N$ is  a free
module. By results of Fantechi, Mann, and Nironi
\cite[Lemma 7.15]{FantechiMannNironi:2010} this is equivalent to the geometric
condition that the associated toric DM stack has no global
stabilizers.

\subsection{Toric DM stacks from stacky fans and polytopes} 
\label{subsec:QuotientConstructionOfDMstacks}
Recall (as in \cite{BCS05}) that given a stacky fan
$(N,\Sigma,\beta)$, the corresponding DM stack may be constructed as a quotient stack $[Z_\Sigma/G]$ as follows.   As with classical toric varieties, the fan $\Sigma$ determines an ideal $J_\Sigma$ generated by the monomials $\prod_{\rho_i \not\subset \sigma} z_i \in \C[z_1, \ldots, z_n]$ corresponding to the cones $\sigma$ in $\Sigma$. Let $Z_\Sigma$ denote the complement $\C^n \setminus V(J_\Sigma)$ of the vanishing locus of $J_\Sigma$.  
Note that $Z_\Sigma$ is the complement of a union of coordinate subspaces of complex codimension at least 2; therefore, $Z_\Sigma$ is simply connected.
Next, we recall a certain group action on $Z_\Sigma$.

Choose a free
resolution
$$
0\to \Z^r \stackrel{Q}{\longrightarrow} \Z^{d+r} \to N \to 0
$$  
of the $\Z$-module $N$, and let $B:\Z^n \to \Z^{d+r}$ be a lift of
$\beta$.  With these choices, define the \emph{dual group} $\DG(\beta) = (\Z^{n+r})^\dual/ \im [B\,
Q] ^\dual$ where $[B\,Q]:\Z^{n+r}=\Z^{n}\oplus \Z^r \to \Z^{d+r}$
denotes the map whose restrictions to the first and second summands
are $B$ and $Q$, respectively.  Let $\beta^\vee: (\Z^n)^\dual \to
\DG(\beta)$ be the composition of the inclusion $(\Z^n)^\dual \to
(\Z^{n+r})^\dual$ (into the first $n$ coordinates) and the quotient
map $(\Z^{n+r})^\dual \to \DG(\beta)$. Applying the functor
$\Hom_\Z(-,\C^*)$ to $\beta^\vee$ yields a homomorphism
$G:=\Hom_\Z(\DG(\beta),\C^*) \to (\C^*)^n$, which defines a $G-$action
on $\C^n$, which leaves $Z_\Sigma \subset \C^n$ invariant.  Define $\mathcal{X}(N,\Sigma,\beta)=[Z_\Sigma/G]$.  By Proposition 3.2 in \cite{BCS05}, $\mathcal{X}(N,\Sigma,\beta)$ is a  DM stack.
\medskip

The above construction was adapted to stacky polytopes by Sakai in \cite{Sakai2010}.   As the reader may verify, the DM stack $\mathcal{X}(N,\Delta,\beta)$ obtained from a stacky polytope is a quotient stack $[Z_\Delta/H]$ where $Z_\Delta$ is a retract of $Z_\Sigma$ (cf.  \cite[Lemma 27]{Sakai2010}) equipped with an action of the compact abelian Lie group  $H=\Hom_\Z(\DG(\beta),S^1)$.  Similar to the discussion in the preceding paragraph, $H$ acts on $\C^n$ and the invariant subset $Z_\Delta$ is a certain level set $\mu^{-1}(c)$ of the moment map $\mu:\C^n \to \mathfrak{h}^\dual$ for this $H$-action (where $\mathfrak{h}$ denotes the Lie algebra of $H$).  In particular, the regular value $c$ is determined by the constants $c_1, \ldots, c_n$ appearing in (\ref{equation:polytope is rational}) (see \cite[Lemma 16]{Sakai2010}).  

 \bigskip
 
 \subsection{The fundamental group and inertia homomorphism of a toric DM stack associated to a stacky fan} \label{subsec:toricinertia}

By Corollary \ref{cor:orb-pi-one when X simply connected}, the fundamental group of a toric DM stack $\mathcal{X}(N,\Sigma,\beta)=[Z_\Sigma/G]$ associated to a stacky fan $(N,\Sigma,\beta)$ is $\pi_1(\mathcal{X}(N,\Sigma,\beta),z) \cong G/G_0$, where $G_0$ is the connected component of the identity element.  Using Proposition  \ref{proposition:quotient homomorphism}, we compute the inertia homomorphisms $\omega_z: G_z \to \pi_1(\mathcal{X}(N,\Sigma,\beta),z)\cong G/G_0$, for the various isotropy groups $G_z$ that arise.

In \cite{GHJK:2011}, both the isotropy groups $G_z$ and the quotient $G/G_0$ are described in terms of the stacky fan data, which we summarize next.
The isotropy group of a point in $\mathcal{X}(N,\Sigma,\beta)$ arises as the stabilizer $\Stab(z) \subset G$ of $z\in Z_\Sigma \subset \C^n$.  These stabilizers depend only on the cone $\sigma$ in $\Sigma$ satisfying $\{i\, : \, z_i=0 \}= \{i \, : \, \rho_i \subset \sigma \}$; namely, for such a cone $\sigma$, the corresponding isotropy group $\Gamma_\sigma$ is the kernel of the composition
$$G \stackrel{(\beta^\vee)^*}{\longrightarrow} (\C^*)^n \longrightarrow (\C^*)^{|J_\sigma|},$$
where $J_\sigma= \{j \, : \, \rho_j \not\subset \sigma \}$.  Hence we shall write the inertia homomorphisms  as $\omega_\sigma:\Gamma_\sigma \to \pi_1(\mathcal{X}(N,\Sigma,\beta))$.

 As shown in  \cite{GHJK:2011}, we may identify $\Gamma_\sigma$ with $\Tor(N/N_\sigma)$, the torsion submodule of the quotient $N/N_\sigma$, where $N_\sigma = \mathrm{span}\{ \beta(\epsilon_i) \, : \, \rho_i \subset \sigma\}$.  Moreover, the inclusion $\Gamma_\sigma \to G$ may be modelled by an explicit homomorphism $\gamma_\sigma :  \Tor(N/N_\sigma) \to G$ constructed in \cite{GHJK:2011}.  Additionally, the quotient $G \to G/G_0$, which is obtained by applying $\Hom(-,\C^*)$ to the inclusion of the torsion submodule $\Tor(\DG(\beta) \hookrightarrow \DG(\beta)$,  may also be modelled by an explicit isomorphism $\Hom(\Tor(\DG(\beta)),\C^*) \stackrel{\cong}{\longrightarrow} \coker \beta$.   (See \cite{GHJK:2011} for details.)

 It is then straightforward to verify that the diagram 
 $$
 \xymatrix{
\Tor(N/N_\sigma) \ar@{^{(}->}[r] \ar[d]_{\cong} \ar[dr]^{\gamma_\sigma} & N/N_\sigma \ar[r] & N/\im(\beta)  \\
\Gamma_\sigma \ar@{^{(}->}[r] & \Hom(\DG(\beta),\C^*) \ar[r] & \Hom(\Tor(\DG(\beta)),\C^*) \ar[u]_{\cong}
 }
 $$
 commutes; therefore, the inertia homomorphism may be identified with the composition in the top row, which proves the following.

\begin{proposition} \label{prop:inertia homomorphism for DM stacks}
Let $(N,\Sigma,\beta)$ be a stacky fan and let $\sigma$ be a cone in $\Sigma$. Using the identifications above, the inertia homomorphisms $\omega_\sigma:\Gamma_\sigma \to \pi_1(\mathcal{X}(N,\Sigma,\beta),z)$ may be identified with the composition 
$$
\omega_\sigma:\Tor(N/N_\sigma) \hookrightarrow N/N_\sigma \to \coker \beta.
$$
\end{proposition}

We may apply the above Proposition to characterize global quotients among toric DM stacks in terms of their stacky fan data, giving another proof of  Corollary \ref{cor:globalquotient} below.  By Theorem \ref{thm:TFAE}, 
it follows that $\mathcal{X}(N,\Sigma,\beta)$ is a global quotient if and only if the kernels $\ker \omega_\sigma = \Tor(\im \beta /N_\sigma)$ are trivial for all cones $\sigma$, if and only if $\im \beta/N_\sigma$ is trivial for all \emph{maximal} cones $\sigma$.

\bigskip

\subsection{The universal cover of a toric DM stack associated to a stacky fan} \label{subsec:toriccover}

By Proposition \ref{prop:covering}, the universal cover of the DM stack $\mathcal{X}(N,\Sigma,\beta)$ is $[Z_\Sigma/G_0]$, where $G_0$ is the connected component of the identity element of $G=\Hom_\Z(\DG(\beta),\C^*)$.  Next we describe $G_0$ in terms of the stacky fan $(N,\Sigma,\beta)$.  As we shall see, this can be roughly described as replacing the abelian group $N$ with the image of $\beta$.

Let $N'\subset N$ denote the image of $\beta$, and let $\Sigma'$ be the fan in $N'\otimes \R$ corresponding to $\Sigma$ under the natural isomorphism $N'\otimes \R \cong N\otimes \R$.  Finally, let $\beta':\Z^n \to N'$ be $\beta$ with restricted codomain.  The following lemma is easily verified.

\begin{lemma} Let $(N,\Sigma,\beta)$ be a stacky fan, and let $(N', \Sigma', \beta')$ be defined as above.  Then  $(N', \beta', \Sigma')$  is a stacky fan.
\end{lemma}

\begin{proposition} \label{prop:universal cover of toric stack} Let $(N, \Sigma, \beta)$ be a stacky fan, and let $(N', \Sigma', \beta')$ be defined as above.  Then the toric DM stack $\calX(N', \Sigma',\beta')$ is the universal cover of $\calX(N,\Sigma,\beta)$.
\end{proposition}
\begin{proof} By Proposition \ref{prop:covering}, it suffices to verify that the toric DM stack  $\calX(N', \Sigma',\beta')=[Z_\Sigma/G_0]$, where the group  $G=\Hom_\Z(\DG(\beta),\C^*)$,
which is verified in \cite{GHJK:2011}.
\end{proof}

For a stacky fan $(N,\Sigma,\beta)$, given a cone $\sigma$ in $\Sigma$, let $N_\sigma \subset N$ denote $ \mathrm{span} \{\beta(e_i) \, | \, \rho_i \subset \sigma \}$. 
The following Lemma describes well-known conditions on a stacky fan $(N,\Sigma,\beta)$ that characterize when the toric DM stack $\calX(N,\Sigma,\beta)$ is in fact a smooth (toric) manifold.  The Corollary that follows then immediately characterizes global quotients among toric DM stacks.

\begin{lemma} \label{lemma:when toric stack is smooth} Let $(N, \Sigma,\beta)$ be a stacky fan.  Then the toric DM stack $\calX(N,  \Sigma,\beta)$ is  (equivalent to) a smooth manifold if and only if $N=N_\sigma$ for all maximal cones $\sigma \in \Sigma$.
\end{lemma}
\begin{proof} 
Recall that since $\calX(N,\Sigma,\beta)$ is Deligne-Mumford, it admits an \'etale presentation, and the diagonal map $\Delta:\calX(N,\Sigma,\beta) \to \calX(N,\Sigma,\beta) \times \calX(N,\Sigma,\beta)$ is proper therefore a closed embedding.  By Proposition 74 in \cite{Met03} it suffices to check that all isotropy groups are trivial.
This follows from Theorem 4.2 \cite{GHJK:2011}. 
\end{proof}


\begin{corollary} \label{cor:globalquotient} Let $(N, \Sigma,\beta)$ be a stacky fan.  The following are equivalent:
\begin{enumerate}
\item The toric DM stack $\calX(N, \Sigma, \beta)$ is equivalent to a global quotient  over $\Diff$. 
\item  $N' = N'_\sigma$ for all maximal cones $\sigma$ in $\Sigma$.
\end{enumerate}
\end{corollary}

\begin{remark}  Corollary \ref{cor:globalquotient} is also obtained in
  joint work of the authors with Goldin and Johanssen \cite{GHJK:2011} by working out the combinatorial condition (2) directly from the equivalent condition that the connected component of the identity $G_0\subset G=\Hom(\DG(\beta),\C^*)$ act freely on $Z_\Sigma$.
\end{remark}

\begin{remark} A similar result to Proposition \ref{prop:universal cover of toric stack} was obtained by Poddar and Sarkar for quasi-toric orbifolds, which are effective/reduced orbifolds studied using methods from toric topology  (see Theorem 3.2 in \cite{PoddarSarkar:2010}). 
\end{remark}

We may recast the above in terms of stacky polytopes.  Given a stacky polytope $(N,\Delta, \beta)$,  let $N'=\im \beta$ and  $\beta':\Z^n \to N'$ as before and let $\Delta'$ be the polytope in $(N'\otimes \R)^\dual$ described by
$$
 \Delta' = \bigcap_{i=1}^n \left\{ x \in (N' \otimes  \R)^\dual \,|\,  \langle
x, \beta'(e_{i})\otimes 1 \rangle  \geq -c_i \right\}
$$
where the numbers $c_1, \ldots, c_n$ are the same as those appearing in (\ref{equation:polytope is rational}) for the polytope $\Delta$. This ensures that the corresponding level sets $Z_\Delta$ and $Z_\Delta'$ coincide.  Equivalently, $\Delta' \subset (N'\otimes \R)^\dual$ is the polytope corresponding to $\Delta$ under the dual of the natural identification $N'\otimes \R \cong N\otimes \R$. Analogous to Proposition \ref{prop:universal cover of toric stack},  $\mathcal{X}(N',\Delta',\beta')$ is the universal cover of $\mathcal{X}(N,\Delta,\beta)$.
In addition, the natural covering $p:\mathcal{X}(N',\Delta',\beta') \to \mathcal{X}(N,\Delta,\beta)$ (as in Proposition \ref{prop:covering}) is compatible with the underlying symplectic structures. 

\begin{proposition}
Let $(N,\Delta,\beta)$ be a stacky polytope, $(N',\Delta',\beta')$ be as above, and let $p:\mathcal{X}(N',\Delta',\beta') \to \mathcal{X}(N,\Delta,\beta)$ be the universal covering projection.  If $\omega$ and $\omega'$ denote the symplectic forms on $\mathcal{X}(N,\Delta,\beta)$ and $\mathcal{X}(N',\Delta',\beta')$, respectively, then $p^*\omega =\omega'$.
\end{proposition}
\begin{proof}
Recall that the symplectic form $\omega$ on a toric DM stack arising from a stacky polytope $(N,\Delta,\beta)$ can be identified with the differential 2-form $\xi^*\omega $ on $Z_\Delta$ where $\xi:\und{Z_\Delta} \to [Z_\Delta/G]$ is a presentation (see Proposition 2.9 in \cite{LermanMalkin2009} and Theorem 14 (and the discussion preceding it) in \cite{Sakai2010}). Furthermore, if $\xi':\und{Z_\Delta} \to [Z_\Delta/G_0]$ denotes a presentation for $\mathcal{X}(N',\Delta',\beta')$ we have that $(\xi')^*\omega' = \xi^*\omega$, as they are each the restriction of the same 2-form on $\C^n$.  Therefore, it suffices to verify that $(\xi')^*p^*\omega = \xi^*\omega$.

The natural isomorphism of $G$-bundles $(B\times G_0) \times_{G_0} G \cong B\times G$ (over any base $B$) gives rise to the 2-commutative diagram,
$$
\xymatrix@C=1.5em@R=1.5em{
 & \und{Z_\Delta} \ar[dl]_{\xi'}
 \ar[dr]^{\xi} \drtwocell<\omit>{<2.0>} & \\
 {[Z_{\Delta}/G_0]} \ar[rr]_{p} & & {[Z_\Delta/G]}
}
$$
which shows that $(\xi')^*p^*\omega = \xi^*\omega$.  
Indeed, a differential form $\omega$ on a  stack $\mathcal{Y}$ is an assignment of a differential form $\omega(y) \in \Omega^*(U)$ for every object $y$ over $U$ that is compatible with maps, in the sense that an arrow $x\to y$ over $f:V\to U$ forces $\omega(x) = f^*\omega(y)$.  It follows that $\xi^*\omega$ is a the 2-form assigned to the object $Z_\Delta \leftarrow Z_\Delta \times G \to Z_\Delta$ and that $(\xi')^*p^*\omega$ is the 2-form assigned to the object $Z_\Delta \leftarrow (Z_\Delta \times G_0)\times_{G_0} \times G \to Z_\Delta$. But the natural isomorphism covering the identity between these objects (i.e. the 2-isomorphism in the diagram above) and the compatibility condition forces these 2-forms to coincide.

\end{proof}

Notice that the polytopes $\Delta'$ and  $\Delta$ are the same up to a rescaling of the underlying lattices via the natural identification $N'\otimes \R \cong N\otimes \R$; therefore, their corresponding volumes satisfy the relation $\mathrm{vol}(\Delta') =|N/N'|\,\mathrm{vol}(\Delta)$.  Corollary \ref{cor:vol} below verifies a similar relation among the corresponding symplectic volumes.

\begin{corollary} \label{cor:vol}
Let $(N,\Delta,\beta)$ be a stacky polytope, and let $(N',\Delta',\beta')$ be as above.  The corresponding symplectic volumes satisfy $\mathrm{Vol}(\mathcal{X}(N',\Delta',\beta') )= |\coker \beta| \cdot \mathrm{Vol}(\mathcal{X}(N,\Delta,\beta))$.
\end{corollary}
\begin{proof}
Let $\mathcal{X}=\mathcal{X}(N,\Delta,\beta)$ and $\mathcal{X}' = \mathcal{X}(N',\Delta',\beta')$ and $p:\calX ' \to \calX$ denote the universal covering projection.
Choose an \'etale presentation $\varphi:\und{X_0} \to\calX$ with a partition of unity so that $\mathrm{Vol}(\calX) = \int_{X_0} \omega^{d}$ (see \cite{Behrend04} for details about integration on stacks).  Since $p$ is a covering projection,  the fiber product $\calX' \times_{\calX} X_0 \cong \und{W}$ for some manifold $W$ and $W\to X_0$ is a covering projection with fiber $G/G_0 \cong \coker \beta$, where $G_0$ is the connected component of $G=\Hom_\Z(\DG(\beta),S^1)$. 
 (In fact, as in Lemma  \ref{lemma:associatedbundle}, we may take $W=E_\varphi/G_0$ where $X_0 \leftarrow E_\varphi \to Z_\Delta$ is the object representing $\varphi(\mathrm{id}_{X_0})$.)  
Then we have the following 2-Cartesian diagram
$$
\xymatrix@C=2.0em@R=2.0em{
{W} \ar[d]^{\varphi'} \ar[r]^{p_0} & {X_0} \ar[d]^{\varphi} \\
{\calX'} \ar[r]^{p} & \calX
}
$$
It follows that $\varphi':\und{W}\to \calX'$ is an \'etale presentation for $\calX'$.  Moreover, we may pull back the partition of unity on $X_0$ to $W$.
By the previous proposition, 
$$
\mathrm{Vol}(\calX') = \int_{W} (\varphi')^*p^*\omega = \int_{W} (p_0)^*\varphi^*\omega = \mathrm{deg}(p_0) \cdot \int_{X_0} \varphi^*\omega = |G/G_0| \cdot \mathrm{Vol}(\calX).
$$
\end{proof}

\subsection{Examples} \label{subsec:examples}

We conclude with some examples illustrating the discussion above.

\bigskip

The following class of examples is studied in \cite{GHJK:2011}.

\begin{example}[Labelled sheared simplices]  Let $\mathbf{a}=(a_1, \ldots, a_d)$ be a primitive vector in $N= \Z^{d}$ with $a_i \geq 1$ and let $m_0, \ldots, m_d \in \Z_{>0}$. Let $\Sigma(\mathbf{a})$ be the fan in $N\otimes \R \cong \R^d$ whose rays are generated by $-\mathbf{a}$ and the standard basis vectors.  Note that $\Sigma(\mathbf{a})$ is the normal fan of a \emph{sheared simplex} $\Delta(\mathbf{a})$.
Letting $f_0, \ldots, f_d$ be the standard basis vectors for $\Z^{d+1}$, set $\beta: \Z^{d+1} \to \Z^d$ to be $\beta(f_0)=-m_0\mathbf{a}$ and $\beta(f_j) = m_j e_j$ for $j=1, \ldots, d$.  
It is straightforward to verify that $(N,\Sigma(\mathbf{a}),\beta)$ is the stacky fan associated to the labelled polytope $(\Delta(\mathbf{a}), \{m_j \}_{j=0}^d)$.  (See Figure \ref{fig:shearedsimp}  illustrating a concrete example with $\mathbf{a}=(1,2)$ and labels $m_0=m_1=1$ and $m_2=2$.)

\begin{figure}
\centering
\mbox{
\subfigure[ $\Sigma(\mathbf{a})$ in $N\otimes \R$, with $N'$ indicated by darkened dots.]{
 \def\svgwidth{0.2\textwidth}
  
\immediate\write18{inkscape -z -D –file=images/fan-quotofCP2.svg
–export-pdf=images/fan-quotofCP2.pdf –export-latex}
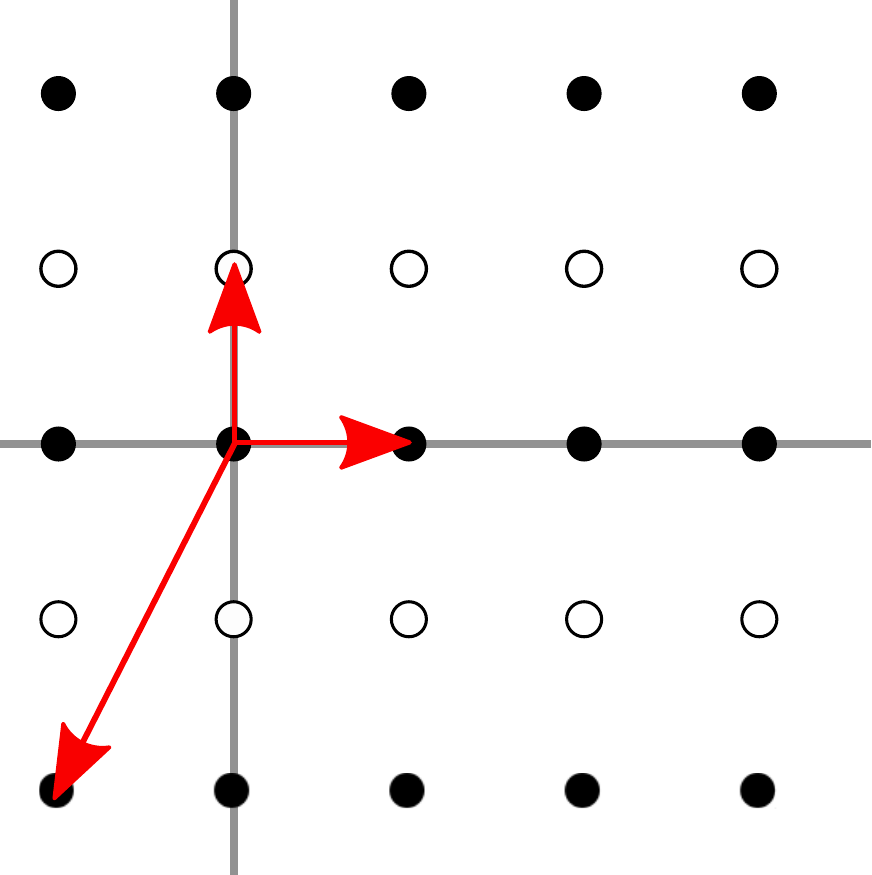
  \label{fig:fanshearedsimp}
 }
 }\quad \quad
   \subfigure[$\Delta(\mathbf{a})$ in $(N\otimes \R)^\dual$, with non-trivial labels indicated.]{
 \def\svgwidth{0.2\textwidth}
  
\immediate\write18{inkscape -z -D –file=images/Ptope-quotofCP2.svg
–export-pdf=images/Ptope-quotofCP2.pdf –export-latex}
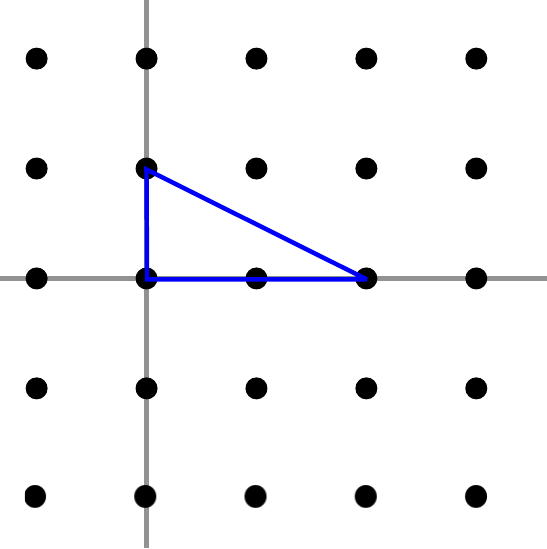
  \label{fig:ptopeshearedsimp}
  
}
\caption{A fan $\Sigma(\mathbf{a}) \subset N\otimes \R$ and its corresponding labelled polytope $
\Delta (\mathbf{a}) \subset (N\otimes \R)^\dual$.} \label{fig:shearedsimp}
\end{figure}

 In \cite{GHJK:2011}, it is shown that the toric DM stack $\mathcal{X}(N,\Sigma(\mathbf{a}),\beta)$ is a global quotient if and only if $m_j=m_0a_j$ for all $j=1, \ldots, d$.  In this case, $N'=\bigoplus_{j=1}^{d} m_j \Z \subset N$ and that under an identification $N'\cong \Z^d$ we find $\Sigma'$ is the fan in $\R^d$ whose rays are generated by $-\sum e_j$ and the standard basis vectors $e_1, \ldots, e_d$ and that $\beta':\Z^{d+1} \to N'\cong \Z^d$ may be expressed by the formulas $\beta'(f_0) = -\sum e_j$, $\beta'(f_j)=e_j$.  That is, $\mathcal{X}(N', \Sigma', \beta') = \CP^d$ so that $\mathcal{X}(N,\Sigma(\mathbf{a}),\beta)$ is a quotient of complex projective space whenever it is a global quotient. (See Figure \ref{fig:shearedsimpcover} illustrating the (stacky) fan and polytope of the universal cover of the symplectic toric DM stack whose fan and polytope appear in Figure \ref{fig:shearedsimp}.)

\begin{figure}
\centering
\mbox{
\subfigure[ $\Sigma(\mathbf{a})'$ in $N'\otimes \R\cong \R^2$ corresponding to the fan  in Fig. \thesection.\ref{fig:fanshearedsimp} .]{
 \def\svgwidth{0.2\textwidth}
  
\immediate\write18{inkscape -z -D –file=images/fan-quotofCP2-cov.svg
–export-pdf=images/fan-quotofCP2-cov.pdf –export-latex}
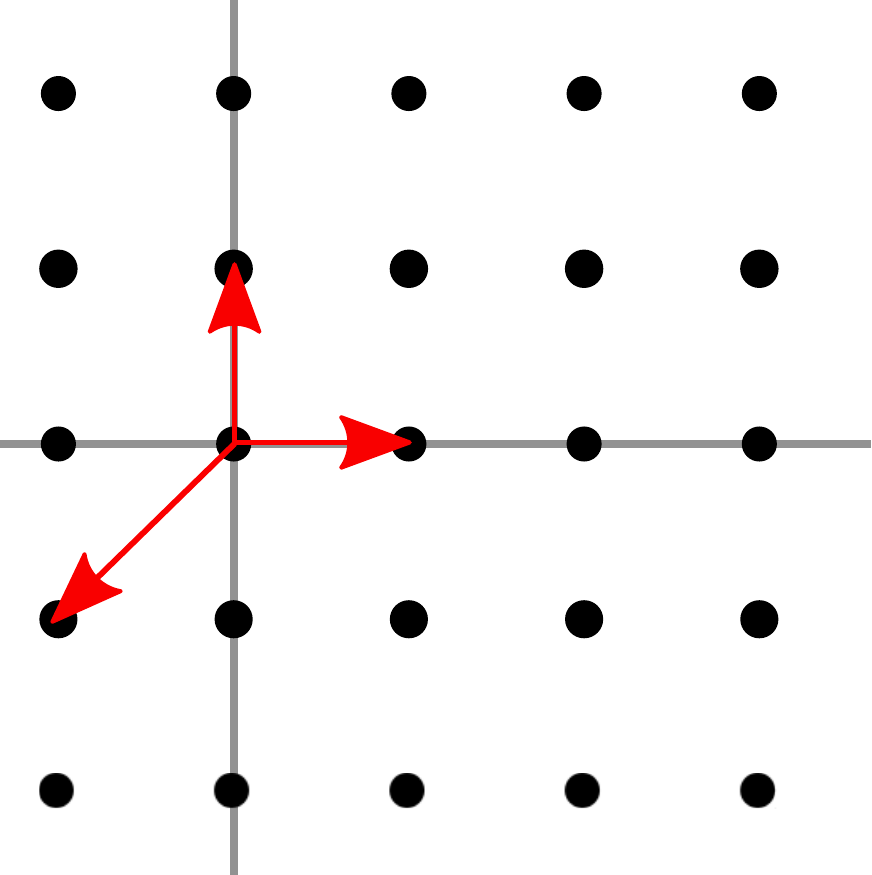
  \label{fig:fanshearedsimpcover}
 }
 }\quad \quad
   \subfigure[$\Delta(\mathbf{a})'$ in $(N'\otimes \R)^\dual \cong \R^2$ corresponding to the labelled polytope in Fig. \thesection.\ref{fig:ptopeshearedsimp}.]{
 \def\svgwidth{0.2\textwidth}
  
\immediate\write18{inkscape -z -D –file=images/Ptope-quotofCP2-cov.svg
–export-pdf=images/Ptope-quotofCP2-cov.pdf –export-latex}
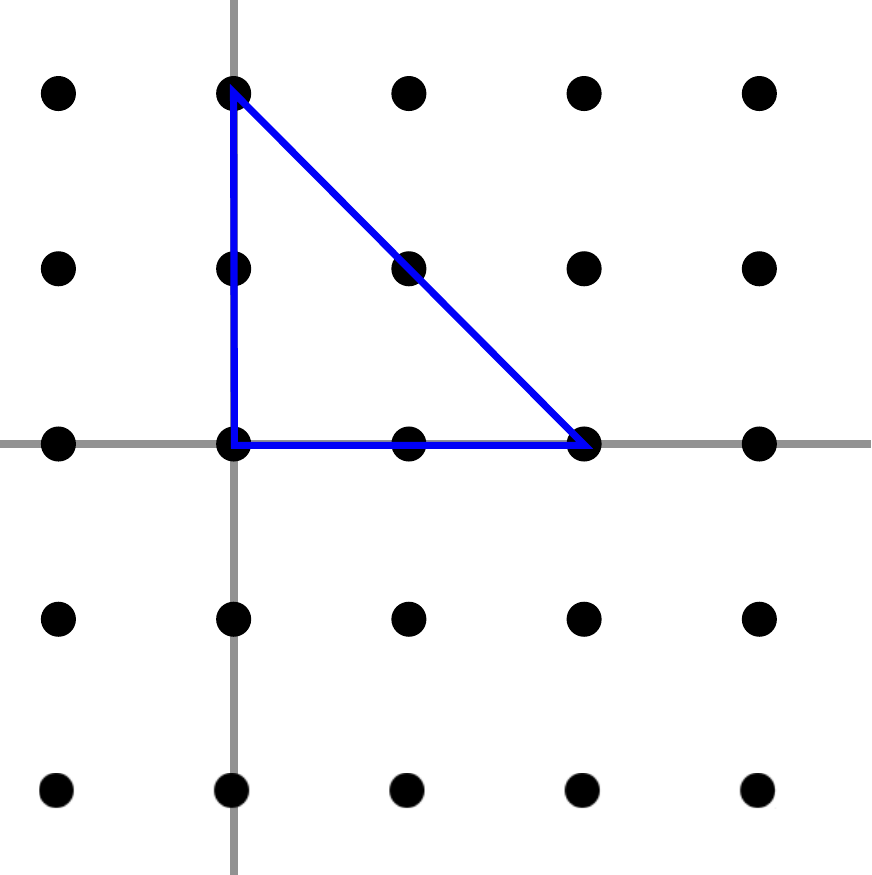
  \label{fig:ptopeshearedsimpcover}
  
}
\caption{The fan $\Sigma(\mathbf{a})'$ and corresponding labelled polytope $
\Delta (\mathbf{a})' $ of the universal cover of the symplectic toric DM stack represented  in Figure \ref{fig:shearedsimp}.  } \label{fig:shearedsimpcover}
\end{figure}

\end{example}

\begin{example}
Let $\mathbf{a}=(a_1,a_2)$ be a primitive vector in $N=\Z^2$ with $a_1,a_2\geq 1$.  Consider  the fan $\Sigma$ in $N \otimes \R \cong \R^2$ with four rays $\rho_1, \ldots, \rho_4$ generated by $-\mathbf{a}$, $e_2$, $e_1$, $-e_2$ respectively, and maximal cones $\sigma_{12}$, $\sigma_{23}$, $\sigma_{34}$, and $\sigma_{41}$, where $\sigma_{ij}$ denotes the two dimensional cone generated by $\rho_i$ and $\rho_j$.
Let $m_1, \ldots, m_4$ be positive integers and let $\beta:\Z^4 \to N$ be
$$
\beta=
\begin{bmatrix}
-m_1a_1 & 0 & m_3 & 0 \\
-m_1a_2 & m_2 & 0 & -m_4 \\
\end{bmatrix}.
$$
Note that the stacky fan $(N,\Sigma,\beta)$ corresponds to a labelled right trapezoid.  (See Figure \thesection.\ref{fig:ptopequotofHbruch} illustrating a concrete example with $\mathbf{a}=(1,2)$.)

By Corollary \ref{cor:globalquotient}, the toric DM stack $\mathcal{X}(N,\Sigma,\beta)$ is a global quotient if and only if $N'=N'_\sigma$ for all maximal cones $\sigma$.  This occurs precisely when $m_1a_1 = m_3$,  $m_2=m_4$, and $m_2|m_1a_2$.  In this case, $N' = m_3 \Z \oplus m_2 \Z \subset N$ and that under an identification $N'\cong \Z^2$ we see that $\Sigma'$ is the fan in $\R^2$ with four rays generated by $-(1,b)$, $e_2$, $e_1$, and $-e_2$, where $b=m_1a_2/m_2$.   Moreover, under this identification 
$$
\beta'=
\begin{bmatrix}
-1 & 0 & 1 & 0 \\
-b & 1 & 0 & -1 \\
\end{bmatrix};
$$
therefore, $\mathcal{X}(N,\Sigma,\beta)$ is a quotient of a Hirzebruch surface whenever it is a global quotient. (See Figure \ref{fig:quotofHbruch}  illustrating the labelled polytope of a global quotient of a Hirzebruch surface.)

\begin{figure}[ht]
\centering
\mbox{
\subfigure[A labelled polytope $\Delta(\mathbf{a})$ in $(N\otimes \R)^\dual \cong \R^2$ corresponding to a stack that is a global quotient.]{
 \def\svgwidth{0.2\textwidth}
  
\immediate\write18{inkscape -z -D –file=images/ptope-quotofHbruch.svg
–export-pdf=images/ptope-quotofHbruch.pdf –export-latex}
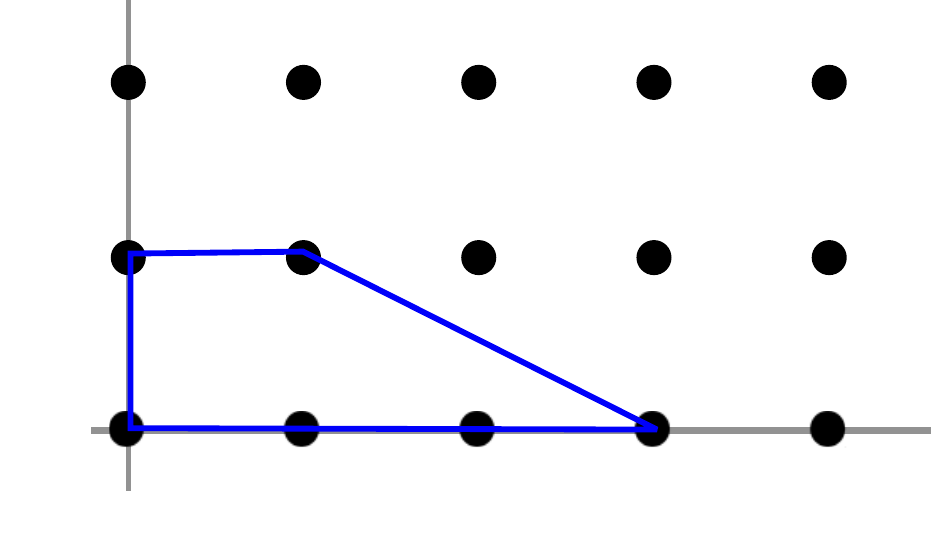
  \label{fig:ptopequotofHbruch}
 }
 }\quad \quad
   \subfigure[The polytope $\Delta(\mathbf{a})'$ in $(N'\otimes \R)^\dual \cong \R^2$  corresponding to the universal cover of the stack associated to \thesection.\ref{fig:ptopequotofHbruch}. ]{
 \def\svgwidth{0.2\textwidth}
  
\immediate\write18{inkscape -z -D –file=images/ptope-quotofHbruch-cov.svg
–export-pdf=images/ptope-quotofHbruch-cov.pdf –export-latex}
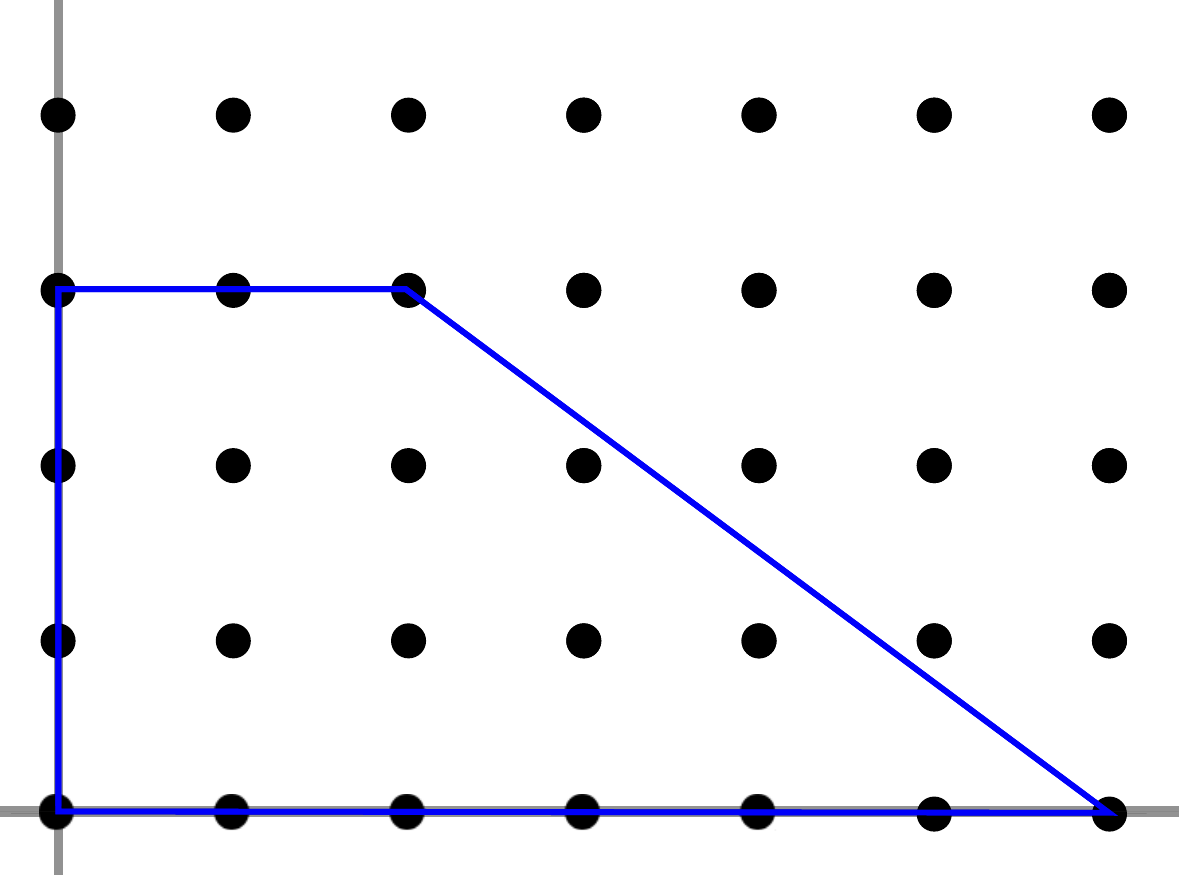
  \label{fig:ptopequotofHbruch-cov}
  
}
\caption{The labelled polytopes $\Delta(\mathbf{a})$ and $\Delta(\mathbf{a})'$ of a
of a quotient of a Hirzebruch surface and its universal cover (a Hirzebruch surface), respectively. } \label{fig:quotofHbruch}
\end{figure}

\end{example}

The remaining three examples consider toric DM stacks with non-trivial global stabilizer (due to the presence of torsion in the abelian group $N$).  The first example exhibits a global quotient with global isotropy, while the last two  illustrate how the condition in Corollary \ref{cor:globalquotient} may fail.

\begin{example}
Let $N=\Z \oplus \Z/2\Z$ and $\Sigma$ be the fan in $N\otimes \R \cong \R$ with rays $\pm 1$.  Let $\beta:\Z^2 \to N$ be given by $ (a,b)\mapsto (2a-2b,a+b \modulo 2)$. Then $N'$ is generated by $(2,1 \modulo 2)$, as is $N'_\sigma$ for each maximal cone $\sigma$.  Therefore, $\mathcal{X}(N,\Sigma,\beta)$ is a global quotient.  Under an identification $N'\cong \Z$, we see that $\beta'(a,b) = a-b$ and that $\mathcal{X}(N',\Sigma',\beta')=\CP^1$.  Moreover, a direct calculation shows that $G/G_0 \cong \coker \beta \cong \Z/4\Z$ so that $\Z/4$ acts on $\CP^1$ with  global stabilizer isomorphic to $\Z/2\Z$ (the torsion submodule of $N$), and  $\mathcal{X}(N,\Sigma,\beta)\cong[\CP^1/(\Z/4\Z)]$.
\end{example}

\begin{example}

Let $(N,\Sigma,\beta)$ be any stacky fan with  $N$ containing a non-trivial torsion subgroup and $\beta$ surjective (e.g. the stacky fan of a weighted projective space with non-trivial global stabilizer).  Then $\mathcal{X}(N,\Sigma,\beta)$ is not equivalent to a global quotient since $N' = N$ has torsion and $N'_\sigma$ is necessarily torsion free for any cone $\sigma$ in the rational simplicial fan $\Sigma$. (cf. the proof  of Theorem 3.1 in \cite{GHJK:2011}). (More generally,  if $N'\subset N$ contains non-trivial torsion, then $\mathcal{X}(N,\Sigma,\beta)$ is not equivalent to a global quotient.)

\end{example}

\begin{example}

Let $N=\Z \oplus \Z/4\Z$ and $\Sigma$ be the fan in $N\otimes \R \cong \R$ with rays $\pm 1$.  Let $\beta: \Z^2 \to N$ be given by $(a,b) \mapsto (a-2b,a+2b \modulo 4)$. Then $N'$ is torsion-free, generated by $(1,1) \in N$; however, $N'_\sigma$ is generated by $(2,2)$ for the cone $\sigma$ generated by $-1$.  By Corollary \ref{cor:globalquotient}, $\mathcal{X}(N,\Sigma,\beta)$ is not equivalent to a global quotient.

\end{example}


\def\cprime{$'$}

\end{document}

%% file: fan-quotofCP2.pdf_tex
\begingroup%
  \makeatletter%
  \providecommand\color[2][]{%
    \errmessage{(Inkscape) Color is used for the text in Inkscape, but the package 'color.sty' is not loaded}%
    \renewcommand\color[2][]{}%
  }%
  \providecommand\transparent[1]{%
    \errmessage{(Inkscape) Transparency is used (non-zero) for the text in Inkscape, but the package 'transparent.sty' is not loaded}%
    \renewcommand\transparent[1]{}%
  }%
  \providecommand\rotatebox[2]{#2}%
  \ifx\svgwidth\undefined%
    \setlength{\unitlength}{250.95bp}%
    \ifx\svgscale\undefined%
      \relax%
    \else%
      \setlength{\unitlength}{\unitlength * \real{\svgscale}}%
    \fi%
  \else%
    \setlength{\unitlength}{\svgwidth}%
  \fi%
  \global\let\svgwidth\undefined%
  \global\let\svgscale\undefined%
  \makeatother%
  \begin{picture}(1,1.00368599)%
    \put(0,0){\includegraphics[width=\unitlength]{fan-quotofCP2.pdf}}%
  \end{picture}%
\endgroup%

%% file: Ptope-quotofCP2.pdf_tex
\begingroup%
  \makeatletter%
  \providecommand\color[2][]{%
    \errmessage{(Inkscape) Color is used for the text in Inkscape, but the package 'color.sty' is not loaded}%
    \renewcommand\color[2][]{}%
  }%
  \providecommand\transparent[1]{%
    \errmessage{(Inkscape) Transparency is used (non-zero) for the text in Inkscape, but the package 'transparent.sty' is not loaded}%
    \renewcommand\transparent[1]{}%
  }%
  \providecommand\rotatebox[2]{#2}%
  \ifx\svgwidth\undefined%
    \setlength{\unitlength}{157.4bp}%
    \ifx\svgscale\undefined%
      \relax%
    \else%
      \setlength{\unitlength}{\unitlength * \real{\svgscale}}%
    \fi%
  \else%
    \setlength{\unitlength}{\svgwidth}%
  \fi%
  \global\let\svgwidth\undefined%
  \global\let\svgscale\undefined%
  \makeatother%
  \begin{picture}(1,1.00365311)%
    \put(0,0){\includegraphics[width=\unitlength]{Ptope-quotofCP2.pdf}}%
    \put(0.46733339,0.4040708){\color[rgb]{0,0,0}\makebox(0,0)[lb]{\smash{$2$}}}%
  \end{picture}%
\endgroup%

%% file: fan-quotofCP2-cov.pdf_tex
\begingroup%
  \makeatletter%
  \providecommand\color[2][]{%
    \errmessage{(Inkscape) Color is used for the text in Inkscape, but the package 'color.sty' is not loaded}%
    \renewcommand\color[2][]{}%
  }%
  \providecommand\transparent[1]{%
    \errmessage{(Inkscape) Transparency is used (non-zero) for the text in Inkscape, but the package 'transparent.sty' is not loaded}%
    \renewcommand\transparent[1]{}%
  }%
  \providecommand\rotatebox[2]{#2}%
  \ifx\svgwidth\undefined%
    \setlength{\unitlength}{250.95bp}%
    \ifx\svgscale\undefined%
      \relax%
    \else%
      \setlength{\unitlength}{\unitlength * \real{\svgscale}}%
    \fi%
  \else%
    \setlength{\unitlength}{\svgwidth}%
  \fi%
  \global\let\svgwidth\undefined%
  \global\let\svgscale\undefined%
  \makeatother%
  \begin{picture}(1,1.00368599)%
    \put(0,0){\includegraphics[width=\unitlength]{fan-quotofCP2-cov.pdf}}%
  \end{picture}%
\endgroup%

%% file: Ptope-quotofCP2-cov.pdf_tex
\begingroup%
  \makeatletter%
  \providecommand\color[2][]{%
    \errmessage{(Inkscape) Color is used for the text in Inkscape, but the package 'color.sty' is not loaded}%
    \renewcommand\color[2][]{}%
  }%
  \providecommand\transparent[1]{%
    \errmessage{(Inkscape) Transparency is used (non-zero) for the text in Inkscape, but the package 'transparent.sty' is not loaded}%
    \renewcommand\transparent[1]{}%
  }%
  \providecommand\rotatebox[2]{#2}%
  \ifx\svgwidth\undefined%
    \setlength{\unitlength}{250.95bp}%
    \ifx\svgscale\undefined%
      \relax%
    \else%
      \setlength{\unitlength}{\unitlength * \real{\svgscale}}%
    \fi%
  \else%
    \setlength{\unitlength}{\svgwidth}%
  \fi%
  \global\let\svgwidth\undefined%
  \global\let\svgscale\undefined%
  \makeatother%
  \begin{picture}(1,1.00368599)%
    \put(0,0){\includegraphics[width=\unitlength]{Ptope-quotofCP2-cov.pdf}}%
  \end{picture}%
\endgroup%

%% file: ptope-quotofHbruch.pdf_tex
\begingroup%
  \makeatletter%
  \providecommand\color[2][]{%
    \errmessage{(Inkscape) Color is used for the text in Inkscape, but the package 'color.sty' is not loaded}%
    \renewcommand\color[2][]{}%
  }%
  \providecommand\transparent[1]{%
    \errmessage{(Inkscape) Transparency is used (non-zero) for the text in Inkscape, but the package 'transparent.sty' is not loaded}%
    \renewcommand\transparent[1]{}%
  }%
  \providecommand\rotatebox[2]{#2}%
  \ifx\svgwidth\undefined%
    \setlength{\unitlength}{268.19419683bp}%
    \ifx\svgscale\undefined%
      \relax%
    \else%
      \setlength{\unitlength}{\unitlength * \real{\svgscale}}%
    \fi%
  \else%
    \setlength{\unitlength}{\svgwidth}%
  \fi%
  \global\let\svgwidth\undefined%
  \global\let\svgscale\undefined%
  \makeatother%
  \begin{picture}(1,0.57695872)%
    \put(0,0){\includegraphics[width=\unitlength]{ptope-quotofHbruch.pdf}}%
    \put(0.54669838,0.25440265){\color[rgb]{0,0,0}\makebox(0,0)[lb]{\smash{$2$}}}%
    \put(-0.00430456,0.19745312){\color[rgb]{0,0,0}\makebox(0,0)[lb]{\smash{$2$}}}%
    \put(0.1618516,0.35102781){\color[rgb]{0,0,0}\makebox(0,0)[lb]{\smash{$3$}}}%
    \put(0.31027065,0.00796348){\color[rgb]{0,0,0}\makebox(0,0)[lb]{\smash{$3$}}}%
  \end{picture}%
\endgroup%

%% file: ptope-quotofHbruch-cov.pdf_tex
\begingroup%
  \makeatletter%
  \providecommand\color[2][]{%
    \errmessage{(Inkscape) Color is used for the text in Inkscape, but the package 'color.sty' is not loaded}%
    \renewcommand\color[2][]{}%
  }%
  \providecommand\transparent[1]{%
    \errmessage{(Inkscape) Transparency is used (non-zero) for the text in Inkscape, but the package 'transparent.sty' is not loaded}%
    \renewcommand\transparent[1]{}%
  }%
  \providecommand\rotatebox[2]{#2}%
  \ifx\svgwidth\undefined%
    \setlength{\unitlength}{339.45bp}%
    \ifx\svgscale\undefined%
      \relax%
    \else%
      \setlength{\unitlength}{\unitlength * \real{\svgscale}}%
    \fi%
  \else%
    \setlength{\unitlength}{\svgwidth}%
  \fi%
  \global\let\svgwidth\undefined%
  \global\let\svgscale\undefined%
  \makeatother%
  \begin{picture}(1,0.74200913)%
    \put(0,0){\includegraphics[width=\unitlength]{ptope-quotofHbruch-cov.pdf}}%
  \end{picture}%
\endgroup%

%% file: sepstabilizer2.bbl
\begin{thebibliography}{10}

\bibitem{AdemLeidaRuan:2007}
A.~Adem, J.~Leida, and Y.~Ruan.
\newblock {\em Orbifolds and stringy topology}, volume 171 of {\em Cambridge
  Tracts in Mathematics}.
\newblock Cambridge University Press, Cambridge, 2007.

\bibitem{AdeRua03}
A.~Adem and Y.~Ruan.
\newblock Twisted orbifold {$K$}-theory.
\newblock {\em Comm. Math. Phys.}, 237(3):533--556, 2003, math.AT/0107168.

\bibitem{Behrend04}
K.~Behrend.
\newblock Cohomology of stacks.
\newblock In {\em Intersection theory and moduli}, volume~19 of {\em ICTP Lect.
  Notes}, pages 249--294. Abdus Salam International Centre for Theoretical
  Physics, Trieste, 2004.

\bibitem{BCEFFGK-stacks}
K.~Behrend, B.~Conrad, D.~Edidin, B.~Fantechi, W.~Fulton, L.~G\"ottssche, and
  A.~Kresch.
\newblock Algebraic stacks,
  \url{http://www.math.uzh.ch/index.php?pr_vo_det&key1=1287&key2=580&no_cache=%
1}.
\newblock in progress.

\bibitem{BCS05}
L.~A. Borisov, L.~Chen, and G.~G. Smith.
\newblock The orbifold {C}how ring of toric {D}eligne-{M}umford stacks.
\newblock {\em J. Amer. Math. Soc.}, 18(1):193--215 (electronic), 2005.

\bibitem{Edidin:2003}
D.~Edidin.
\newblock What is a stack?
\newblock {\em Notices Amer. Math. Soc.}, 50(4):458–--459, 2003.

\bibitem{Fantechi:2001}
B.~Fantechi.
\newblock Stacks for everybody.
\newblock In {\em European {C}ongress of {M}athematics, {V}ol. {I}
  ({B}arcelona, 2000)}, volume 201 of {\em Progr. Math.}, pages 349--359.
  Birkh\"auser, Basel, 2001.

\bibitem{FG03}
B.~Fantechi and L.~G{\"o}ttsche.
\newblock Orbifold cohomology for global quotients.
\newblock {\em Duke Math. J.}, 117(2):197--227, 2003.

\bibitem{FantechiMannNironi:2010}
B.~Fantechi, E.~Mann, and F.~Nironi.
\newblock Smooth toric {D}eligne-{M}umford stacks.
\newblock {\em J. Reine Angew. Math.}, 648:201--244, 2010.

\bibitem{GeraschenkoSatriano:2011a}
A.~Geraschenko and M.~Satriano.
\newblock Toric stacks {I}: The theory of stacky fans, July 2011,
  http://arxiv.org/abs/1107.1906.

\bibitem{GeraschenkoSatriano:2011b}
A.~Geraschenko and M.~Satriano.
\newblock Toric stacks {II}: Intrinsic characterization of toric stacks, July
  2011, http://arxiv.org/abs/1107.1907.

\bibitem{GHJK:2011}
R.~Goldin, M.~Harada, D.~Johannsen, and D.~Krepski.
\newblock On stacky fans and their isotropy.
\newblock in preparation.

\bibitem{Husemoller:1994}
D.~Husemoller.
\newblock {\em Fibre bundles}, volume~20 of {\em Graduate Texts in
  Mathematics}.
\newblock Springer-Verlag, New York, third edition, 1994.

\bibitem{Iwanari:2009b}
I.~Iwanari.
\newblock Logarithmic geometry, minimal free resolutions and toric algebraic
  stacks.
\newblock {\em Publ. Res. Inst. Math. Sci.}, 45(4):1095--1140, 2009.

\bibitem{Kaw73}
T.~Kawasaki.
\newblock Cohomology of twisted projective spaces and lens complexes.
\newblock {\em Math. Ann.}, 206:243--248, 1973.

\bibitem{Lerman:2010}
E.~Lerman.
\newblock Orbifolds as stacks?
\newblock {\em Enseign. Math. (2)}, 56(3-4):315--363, 2010.

\bibitem{LermanMalkin2009}
E.~Lerman and A.~Malkin.
\newblock Hamiltonian group actions on symplectic {D}eligne-{M}umford stacks
  and toric orbifolds.
\newblock {\em Advances in Mathematics}, 229(2):984--1000, January 2012,
  http://arxiv.org/abs/0908.0903.

\bibitem{LT97}
E.~Lerman and S.~Tolman.
\newblock Hamiltonian torus actions on symplectic orbifolds and toric
  varieties.
\newblock {\em Trans. Amer. Math. Soc.}, 349(10):4201--4230, 1997.

\bibitem{Met03}
D.~Metzler.
\newblock {Topological and Smooth Stacks}, June 2003,
  http://arxiv.org/abs/math/0306176.

\bibitem{Noohi:2005}
B.~Noohi.
\newblock Foundations of topological stacks {I},
  http://arxiv.org/abs/math/0503247.

\bibitem{Noohi:2004}
B.~Noohi.
\newblock Fundamental groups of algebraic stacks.
\newblock {\em J. Inst. Math. Jussieu}, 3(1):69--103, 2004.

\bibitem{PoddarSarkar:2010}
M.~Poddar and S.~Sarkar.
\newblock On quasitoric orbifolds.
\newblock {\em Osaka J. Math.}, 47(4):1055--1076, 2010.

\bibitem{Sakai2010}
H.~Sakai.
\newblock The symplectic {D}eligne-{M}umford stack associated to a stacky
  polytope.
\newblock {\em Results in Mathematics}, March 2012,
  http://arxiv.org/abs/1009.3547.

\bibitem{Satake:1956}
I.~Satake.
\newblock On a generalization of the notion of manifold.
\newblock {\em Proc. Nat. Acad. Sci. U.S.A.}, 42:359--363, 1956.

\end{thebibliography}
